\documentclass[10pt]{amsart}
\usepackage{mathtools,amsthm,amssymb}
\usepackage[utf8]{inputenc}
\usepackage{graphicx}
\usepackage{geometry}
\usepackage{verbatim}
\usepackage{hyperref}
\usepackage{charter,eulervm}
\usepackage[mathscr]{euscript}
\usepackage{csquotes}
\usepackage{color}
\usepackage{tikz-cd}
\usepackage{setspace}

\swapnumbers
\theoremstyle{plain}
\newtheorem{theorem}{Theorem}[subsection]
\newtheorem*{theorem*}{Theorem}
\newtheorem{corollary}[theorem]{Corollary}
\newtheorem{proposition}[theorem]{Proposition}

\newtheorem{conjecture}[theorem]{Conjecture}
\newtheorem{lemma}[theorem]{Lemma}

\theoremstyle{definition}

\newtheorem*{definition*}{Definition}

\newtheorem{example}[theorem]{Example}

\theoremstyle{remark}

\newtheorem*{remarks*}{Remarks}

\newtheorem*{notation*}{Notation}

\DeclareMathOperator{\coz}{coz}
\DeclareMathOperator{\con}{con}
\DeclareMathOperator{\cl}{cl}

\DeclareMathOperator{\st}{uf}
\DeclareMathOperator{\clop}{clop}
\DeclareMathOperator{\uf}{uf}
\DeclareMathOperator{\oo}{\mathfrak{o}}

\newcommand{\qtq}[1]{\quad\text{#1}\quad}
\newcommand{\downset}[1]{\left\downarrow{#1}\right\downarrow}

\newcommand{\ol}[1]{\overline{#1}}

\newcommand{\mbf}[1]{\mathbf{#1}}
\newcommand{\mbb}[1]{\mathbb{#1}}
\newcommand{\mcal}[1]{\mathcal{#1}}
\newcommand{\mfrak}[1]{\mathfrak{#1}}

\newcommand{\sbw}[2]{\smashoperator[#1]{\bigwedge_{#2}}}
\newcommand{\sbv}[2]{\smashoperator[#1]{\bigvee_{#2}}}

\newcommand{\setof}[2]{\left\{\,#1 \mid #2\,\right\}}
\newcommand{\map}[3]{#1 \colon #2 \to #3}

\newcommand{\UC}{\mcal{UC}}
\newcommand{\simple}[1]{\sigma{#1}}
\newcommand{\clr}[1]{\delta(#1)}
\newcommand{\exR}{\ol{\mbb{R}}}

\newcommand{\SC}{\mbf{Cmp}}
\newcommand{\open}[1]{\oo(#1)}
\begin{document}
\title[Structure of truncs]{Structural aspects of \\truncated archimedean vector lattices:\\ good sequences, simple elements}
\author{Richard N. Ball}
\address[Ball]{Department of Mathematics, University of Denver, Denver, CO 80208, U.S.A.}
\email[Ball]{rball@du.edu}
\date{\today}
\thanks{File name: Simple Truncs Good Sequences.tex}

\begin{abstract}
The truncation operation facilitates the articulation and analysis of several aspects of the structure of archimedean vector lattices; we investigate two such aspects in this article. We refer to archimedean vector lattices equipped with a truncation as \emph{truncs}.  

In the first part of the article we review the basic definitions, state the (pointed) Yosida Representation Theorem for truncs, and then prove a representation theorem which subsumes and extends the (pointfree) Madden Representation Theorem. The proof has the virtue of being much shorter than the one in the literature, but the real novelty of the theorem lies in the fact that the topological data dual to a given trunc $G$ is a (localic) compactification, i.e., a dense pointed frame surjection $\map{q}{M}{L}$ out of a compact regular pointed frame $M$. The representation is an amalgam of the Yosida and Madden representations; the compact frame $M$ is sufficient to describe the behavior of the bounded part $G^*$ of $G$ in the sense that $\widetilde{G}^*$ separates the points of the compact Hausdorff pointed space $X$ dual to $M$, while the frame $L$ is just sufficient to capture the behavior of the unbounded part of $G$ in $\mcal{R}_0 L$. 

The truncation operation lends itself to identifying those elements of a trunc which behave like characteristic functions, and in the second part of the article we characterize in several ways those truncs composed of linear combinations of such elements. Along the way, we show that the category of such truncs is equivalent to the category of pointed Boolean spaces, and to the category of generalized Boolean algebras. 

The short third part contains a characterization of the kernels of truncation homomorphisms in terms of pointwise closure. In it we correct an error in the literature. 
\end{abstract}

\maketitle

\tableofcontents

\section{Introduction}\label{Sec:Intro}

The need for the truncation operation on $\mbb{R}^X$ for the purposes of measure theory was pointed out by M. H. Stone in \cite{Stone:1948}, and, under the name \enquote*{Stone's condition,} this operation is assumed in many measure theory texts (\cite{Dudley:2003}, \cite{Fremlin:1974}). Moreover, a truncation is a weakening of the requirement of a weak unit, and as such it is all that is necessary to render the representation of a vector lattice canonical. Based on an axiomatic formulation of the truncation operation, this representation theory has recently been worked out in \cite{Ball:2014.1} and \cite{Ball:2014.2}. The article at hand continues this development, and makes frequent reference to the latter two. 

Part \ref{Part:Rep} concerns the representation of truncs. It begins in Section \ref{Sec:YosRep} with a review of the basic constructs to fix notation, followed by a statement of the classical Yosida Representation Theorem \ref{Thm:16}. We take up the pointfree representation in Section \ref{Sec:MaddenRep}, beginning with a review of the trunc operations in $\mcal{R}_0 L$, and culminating in Theorem \ref{Thm:3}. The proof of this theorem is conceptually clear, technically uncomplicated, and short. A given trunc $G$ is shown to be isomorphic to a trunc in $\mcal{E}_0 q$, where $\map{q}{M}{L}$ is a compactification, and $\mcal{E}_0 q$ is the family of those pointed frame maps $g \in \mcal{R}_0 L$ for which there exists another pointed frame map $\map{g'}{\mcal{O}_* \exR}{M}$ such that $g \circ p = q \circ g'$. Here $p$ stands for the pointed frame map of the inclusion $\mbb{R} \to \exR$. 

Both uniform and pointwise convergence have elegant formulations in terms of the truncation operation, and Section \ref{Sec:PtwiseCon} reviews them. The connection between the two convergences is Dini's Theorem \ref{Thm:18}. In Section \ref{Sec:TruncSeq} we take up the truncation sequences used by Hager in his treatment of *-maximal $\mbf{W}$-objects (\cite{Hager:2013}), and in Proposition \ref{Prop:27} we show them to be essentially equivalent to the good sequences which Mundici used to show the equivalence between unital $\ell$-groups and $MV$-algebras (\cite{Mundici:1986}). We conclude Section \ref{Sec:TruncSeq} by characterizing the functions of $\mcal{E}_0 q$ as being those elements $g \in \mcal{R}_0 L$ such that each truncation $g \wedge n$ factors through $q$ (Proposition \ref{Prop:1}).

Part \ref{Part:Rep} concludes with Section \ref{Sec:E0q}. We show that for a compactification $\map{q}{M}{L}$, $\mcal{E}_0 q$ is a trunc iff every cozero element $u \in M$ for which $q(u) = \top$ is $C^*$-embedded (Proposition \ref{Prop:5}). Furthermore, we show $\mcal{E}_0 q = \mcal{R}_0 L$ iff $q$ is the compact regular coreflection, i.e., iff $q$ is the \v{C}ech-Stone compactification (Proposition \ref{Prop:9}). We conclude Part \ref{Part:Rep} with an important question in the form of Conjecture \ref{Con:Snug}.

The truncation operation makes it easy to identify those elements which behave like characteristic functions (Proposition \ref{Prop:4}). In analysis, the term simple function is often used for linear combinations of characteristic functions, and we use that term here, both for the elements and for the truncs composed of them. Under the name \emph{Specker groups,} these objects have received a good deal of attention in the ordered algebra literature (see \cite[p.\ 385] {Darnel:1994}). Part \ref{Part:Simple} presents several characterizations of simple truncs: Theorem \ref{Thm:1} in terms of locally constant functions, Theorem \ref{Thm:12} in terms of functions bounded away from $0$, and Theorem \ref{Thm:13} in terms of hyperarchimedean truncs.

As a matter of independent interest, we show in Section \ref{Sec:EqCat} that the following four categories are equivalent: the category $\mbf{sT}$ of simple truncs, the category $\mbf{gBa}$ of generalized Boolean algebras, the category $\mbf{iBa}$ of idealized Boolean algebras, and the category $\mbf{zdK}_*$ of Boolean pointed spaces, i.e., zero dimensional compact Hausdorff pointed spaces. (See Theorems \ref{Thm:14} and \ref{Thm:15}.)  

Part \ref{Part:TruncKer} is a brief coda. In it the fundamental lemma from the literature characterizing the kernels of truncation homomorphisms receives a correction (Lemma \ref{Lem:20}), and the lemma is then used to show that such kernels are also characterized by the property of being pointwise closed (Proposition \ref{Prop:30}). 

\part{Representation of truncated archimedean vector lattices\label{Part:Rep}}

\section{The Yosida representation\label{Sec:YosRep}}

In this section we introduce the primary constructs to fix notation, and then outline the fundamental Yosida representation Theorem \ref{Thm:16}. 
  
\subsection{Pointed spaces and frames}
We introduce the categories which constitute the setting for the geometrical aspects of our investigations.   

\begin{definition*}[pointed space]
	A \emph{pointed space} is an object of the form $(X, *)$, where $X$ is a Tychonoff topological space and $*$ is a designated point of $X$. A \emph{continuous pointed function} $\map{f}{(Y, *_Y)}{(X, *_X)}$ is a continuous function $\map{f}{Y}{X}$ such that $f(*_Y) = *_X$. We denote the category of pointed spaces with continuous pointed functions by $\mathbf{Sp}_*$. We denote the full subcategory consisting of the compact Hausdorff pointed spaces by $\mathbf{K}_*$. We denote the full subcategory of zero dimensional compact Hausdorff pointed spaces, aka Boolean pointed spaces, by $\mbf{zdK}_*$.
\end{definition*}

The pointfree counterpart of a pointed space is a pointed frame. A fuller development of this topic, together with proofs, can be found in Section 4 of \cite{Ball:2014.2}.

\begin{definition*}[pointed frame]
	A \emph{pointed frame} is a pair $(L,*_L)$, where $L$ is a frame and $*_L$ is a point of $L$, i.e., a frame map $*_L \colon L \to 2 \equiv \{\bot, \top\}$. A \emph{pointed frame homomorphism $f \colon (L,*_L) \to (M,*_M)$} is a frame homomorphism $f \colon L \to M$ which commutes with the points, i.e., $*_M \circ f = *_L$. 
	We denote the category of pointed frames with their homomorphisms by $\mbf{Frm}_*$.
\end{definition*}

Of particular importance is the \emph{pointed frame of the reals $\mcal{O}_* \mbb{R} \equiv (\mcal{O} \mbb{R}, *_0)$}, where $\mcal{O} \mbb{R}$ designates the frame of open subsets of the real numbers and $*_0 \colon \mcal{O} \mbb{R} \to 2$ is the frame map corresponding to the insertion of the real number $0$ into $\mbb{R}$, i.e.,
\[
*_0(U)
= \begin{cases}
\top  & \text{if $0 \in U$}\\
\bot  & \text{if $0 \notin U$}
\end{cases},
\qquad U \in \mcal{O} \mbb{R}.
\] 

The connection between the categories $\mbf{Frm}_*$ and $\mbf{Sp}_*$ is provided by the basic adjunction given by the functors $\map{\mcal{S}_*} {\mbf{Frm}_*} {\mbf{Sp}_*}$ and $\map{\mcal{O}_*}  {\mbf{Sp}_*} {\mbf{Fr}_*}$. The functor $\mcal{S}_*$ assigns to a pointed frame its spectrum, or space of points. The functor $\mcal{O}_*$ assigns to a pointed space $(X, *_X)$ the pointed frame $(\mcal{O}X, *)$, where $\mcal{O}X$ is the frame of open sets of $X$ and $*$ is the frame map of the point insertion $* \to X$.

\subsection{Truncated vector lattices}\label{Sec:Truncation}
We introduce the categories which will provide the setting for the algebraic aspects of our investigation. The underlying motivation is that truncation is a generalization of a weak order unit. 

\begin{definition*}[truncation]
	A \emph{truncation} on an archimedean vector lattice $G$ is a unary operation on the positive cone $G^+$, to be designated $g \mapsto \ol{g}$, satisfying the following axioms for all $h,g \in G^+$.
	\begin{enumerate}
	\item[($\mfrak{T}1$)] 
	$g \wedge \ol{h} \leq \ol{g} \leq g$.

	\item[($\mfrak{T}2$)]
	If $\ol{g} = 0$ then $g = 0$.
	
	\item[($\mfrak{T}3$)]
	If $ng = \ol{ng}$ for all $n$ then $g = 0$.
	\end{enumerate}
	The range $\setof{\ol{g}}{g \in G^+}$ of the truncation is designated by $\ol{G}$.  
\end{definition*}

The symbol $g \ominus 1$ occurs frequently as an abbreviation for $g - \ol{g}$, as does the symbol $g \ominus r$ as an abbreviation for $r(g/r \ominus 1)$, $0 < r \in \mbb{R}$. The expression $g \ominus 0$ is taken to represent $g$. Likewise the symbol $g \wedge r$ occurs frequently as an abbreviation for $r\ol{g/r}$, $0 < r \in \mbb{R}$. The appearance of the symbols $1$ or $r$ here is formal; it is not implied that the trunc contains elements named $1$ or $r$.

\begin{definition*}[trunc]
	A \emph{truncated vector lattice}, or more concisely a \emph{trunc}, is an archimedean vector lattice $G$ equipped with a truncation $g \mapsto \ol{g}$. A \emph{truncation homomorphism} is a vector lattice homomorphism $\theta \colon G \to H$ which preserves the truncation, i.e., $\theta(\ol{g}) = \ol{\theta(g)}$ for all $g \in G^+$. We denote the category of truncs with truncation homomorphisms by $\mbf{T}$.
\end{definition*} 

\begin{definition*}[$\mcal{C}_0 X$]
	The most classical truncs are those of the form 
	\[
		\mcal{C}_0 X
		\equiv \setof{\map{\tilde{g}}{X}{\mbb{R}}}{\text{$\tilde{g}$ is continuous and $\tilde{g}(*) = 0$}},
	\]
	where $(X, *)$ is a pointed Tychonoff space. In $\mcal{C}_0 X$ or any of its variants, truncation is always given by the formula
	\[
		\ol{\tilde{g}}(x) \equiv \tilde{g}(x) \wedge 1, \qquad x \in X.
	\]
\end{definition*}

The well known category $\mbf{W}$ of archimedean vector lattices with designated weak order units can be identified with a non-full subcategory of $\mbf{T}$.  

\begin{definition*}[$\mbf{W}$]
	 A trunc $G$ is called \emph{unital} if it possesses an element $u$, called its \emph{unit}, such that $\ol{g} = g \wedge u$ for all $g \in G^+$. We denote the non-full subcategory of $\mbf{T}$ comprised of the unital truncs by $\mbf{W}$. Its objects are of the form $(G,u)$, where $G$ is an archimedean trunc with unit $u$, and its morphisms are the truncation homomorphisms $\theta \colon (G,u_G) \to (H,u_H)$ such that $\theta(u_G) = u_H$.    
\end{definition*}

\begin{theorem}[\cite{Ball:2014.2}, 6.1.3]
	$\mbf{W}$ is a non-full monoreflective subcategory of $\mbf{T}$. 
\end{theorem}

\subsection{The Yosida representation} \label{Subsec:YosRepTruncs}

Let $\exR \equiv \mbb{R} \cup \{\pm\infty\}$ designate the extended real numbers.


\begin{definition*}[$\mcal{D}_0 X$]
	Let $(X,*)$ be a compact Hausdorff pointed space. A continuous function $\map{\tilde{g}}{X}{\exR}$ is said to be \emph{almost finite} if $\tilde{g}^{-1}(\mbb{R})$ is a dense subset of $\exR$. We denote the family of such functions which vanish at $*$ by  
	\[
		\mcal{D}_0 X
		\equiv \setof{\tilde{g}}{\text{$\map{\tilde{g}}{X}{\exR}$ is continuous, almost finite, and $\tilde{g}(*) = 0$}}.
	\] 
\end{definition*}

$\mcal{D}_0 X$ has many of the features of a trunc: it contains the constant $0$ function (and no other constant function) and is closed under pointwise join, meet, and scalar multiplication. It is also closed under the natural truncation $\ol{\tilde{g}} = \tilde{g} \wedge 1$, but it does not contain the constant $1$ function itself. $\mcal{D}_0 X$ also admits the  partial addition given by the rule: 
\[
	\tilde{f} + \tilde{g} = \tilde{h} \iff  \forall x \in \tilde{f}^{-1}(\mbb{R}) \cap \tilde{g}^{-1}(\mbb{R}) \cap \tilde{h}^{-1}(\mbb{R})\ (\tilde{f}(x) + \tilde{g}(x) = \tilde{h}(x)).
\] 
Even though this partial addition can fail to be total, it may nevertheless be the case that a subset $A \subseteq \mcal{D}_0 X$ is closed under all of the aforementioned operations and is therefore  a trunc. In such a case we say that \emph{$A$ is a trunc in $\mcal{D}_0 X$.} Such truncs are universal objects for $\mbf{T}$ in the following sense.

\begin{theorem}[\cite{Ball:2014.1}5.3.6]\label{Thm:16}
	Let $G$ be a trunc.
	\begin{enumerate}
		\item 
		There is a compact Hausdorff pointed space $(X, *)$, a trunc $\widetilde{G}$ in $\mcal{D}_0 X$, and a trunc isomorphism $\map{\nu_G} {G} {\widetilde{G}} = (g \mapsto \tilde{g})$ such that $\widetilde{G}$ separates the points of $X$. The space $(X,*)$ is unique up to isomorphism with respect to its properties, and is referred to as the \emph{Yosida space of $G$}. 
		
		\item 
		For every trunc homomorphism $\map{\theta}{G}{H}$, where $H$ is a trunc with Yosida space $Y \in \mbf{K}_*$, there exists a unique $\mbf{K}_*$-morphism $\map{k}{Y}{X}$ such that $\widetilde{\theta(g)} = \tilde{g} \circ k$ for all $g \in G$.
		\begin{figure}[htb]
			\begin{tikzcd}
				G \arrow{r}{\nu_G} \arrow{d}[swap]{\theta}
				& \mcal{D}_0 X \arrow{d}{\mcal{D}_0 k}
				& X \arrow{r}{\tilde{g}} 
				&\exR\\
				H \arrow{r}{\nu_H} & \mcal{D}_0 Y &Y \arrow{u}{k} \arrow {ur}[swap]{\widetilde{\theta(g)}}&
			\end{tikzcd}	
		\end{figure}	
	\end{enumerate} 
\end{theorem}

\section{The Madden representation of truncs\label{Sec:MaddenRep}}

\subsection{The trunc $\mcal{R}_0 L$}

\begin{definition*}[$\mcal{R}L$, $\mcal{R}_0 L$, $\mcal{D}_0 M$]
	For any frame $L$, we denote the family of frame homomorphisms $\mcal{O} \mbb{R} \to L$ by $\mcal{R} L$. For a pointed frame $(L,*)$, we denote the family of pointed frame homomorphisms $\mcal{O}_* \mbb{R} \to L$ by $\mcal{R}_0 L$. For a compact pointed frame $M$, we denote the family of pointed frame homomorphisms $\map{h'}{\mcal{O}_* \exR}{M}$ such that $h'(-\infty, \infty)$ is a dense element of $M$ by $\mcal{D}_0 M$. Thus if $M = \mcal{O}_* X$ for $X \in \mbf{K}_*$ then $\mcal{D}_0 M = \setof{\tilde{g}^{-1}}{g \in \mcal{D}_0 X}$.
\end{definition*}

$\mcal{R}_0 L$ inherits a natural trunc structure from the trunc $\mbb{R}$, just as $\mcal{R} L$ inherits its structure from the $\mbf{W}$-object $\mbb{R}$. In similar fashion, $\mcal{D}_0 M$ inherits from $\exR$ a scalar multiplication, the lattice operations, and a truncation, together with a partial addition operation. We briefly outline the details of these operations here; the reader wishing a fuller explanation can find it worked out for $\mcal{R} L$ in \cite{BallWalters:2002} or  \cite{BallHager:1991}, and for $\mcal{R}_0 L$ in  \cite{Ball:2014.2}.  The reader may also consult the several papers of Bernhard Banaschewski related to the topic.

Consider an $n$-ary operation on $\mbb{R}$, i.e., a continuous function $w \colon \mbb{R}^n \to \mbb{R}$. Then $w$ gives rise to an $n$-ary operation on $\mcal{R}_0 L$ as follows. A basic open subset of $\mbb{R}^n$ has the form 
\[
\vec{U} 
= (U_1, U_2, \dots , U_n)
=\setof{(x_1,x_2,\dots,x_n)}{x_i \in U_i},
\] 
where $U_i \in \mcal{O} \mbb{R}$ for all $i$. We write $w(\vec{U}) \subseteq V$ to mean that $w(x_1,x_2,\dots,x_n) \in V$ whenever $x_i \in U_i$ for all $i$. And we use $\vec{f}$ to abbreviate $(f_1,f_2,\dots,f_n) \in (\mcal{R}_0 L)^n$. 

\begin{theorem}[\cite{Ball:2018}, 3.2]\label{Thm:4}
	Assume the foregoing notation. The operation induced on $\mcal{R}_0 L$ by the function $w \colon \mbb{R}^n \to \mbb{R}$, also denoted $w$ by abuse of notation, is given by the formula
	\[
	w(\vec{f})(V)
	= \smashoperator[l]{\bigvee_{w(\vec{U}) \subseteq V}} \bigwedge_if_i(U_i),
	\qquad V \in \mcal{O}_* \mbb{R}.
	\]   
\end{theorem}

The function $w \colon \mbb{R}^n \to \mbb{R}$ may come from a truncation term, i.e., an expression built up from variables and constants using the trunc operations. In that case the corresponding function $w$ is obtained by interpreting each operation and constant of the term as the corresponding operation or constant of $\mbb{R}$. For instance, the functions in Lemma \ref{Lem:8} are those associated with the truncation terms $\ol{x}$ and $x \ominus 1$. The interpretations of these terms in $\mbb{R}$ are $x \wedge 1$ and $x - \ol{x} = (x - 1)^+$, respectively; the lemma provides the formulas for the corresponding frame homomorphisms. 

\begin{lemma}[\cite{Ball:2018}, 3.3]\label{Lem:8}
	The following hold for $g \in \mcal{R}_0^+ L$.
	\begin{align*}
	\ol{g}(-\infty,r) 
	&= \begin{cases}
	\top          & \text{if $ r > 1$}\\
	g(-\infty,r)  & \text{if $r \leq 1$}
	\end{cases}&
	\ol{g}(r, \infty) 
	&= \begin{cases}
	\bot          & \text{if $r \geq 1$}\\
	g(r, \infty)  & \text{if $r < 1$}
	\end{cases}\\
	g \ominus 1(r,\infty) 
	&= \begin{cases}
	\top             & \text{if $r < 0$}\\
	g(r + 1, \infty) & \text{if $r \geq 0$}
	\end{cases}&
	g \ominus 1(-\infty, r) 
	&= \begin{cases}
	\bot             & \text{if $r \leq 0$} \\
	g(-\infty, r + 1)& \text{if $r > 0$}
	\end{cases}
	\end{align*}
\end{lemma}

\begin{corollary}[\cite{Ball:2018}, 3.4]
	If an equation in the elementary language of $\mbf{T}$, meaning an expression of the form $\tau_1 = \tau_2$ for $\mbf{T}$-terms $\tau_i$, holds in the trunc $\mbb{R}$ then it also holds in $\mcal{R}_0 L$. In particular, all of the axioms defining a trunc hold in $\mcal{R}_0 L$.  
\end{corollary} 

Truncs of the form $\mcal{R}_0 L$, $L \in \mbf{Frm}_*$, are universal objects for $\mbf{T}$, a result due to Madden and Vermeer for $\mbf{W}$ (\cite{MaddenVermeer:1986}) and to the author for $\mbf{T}$ (\cite{Ball:2014.2}).  Our objective here, however, is Theorem \ref{Thm:3}, a sharper representation which will subsume both of these results. 

\subsection{Compactifications}

Recall that a frame surjection $\map{q}{M}{L}$ is associated with the sublocale $q_*(L) \equiv \setof{q_*(y)}{y \in L}$, where $\map{q_*}{L}{M} = \big(y \mapsto \bigvee \setof{x}{q(x) \leq y}\big)$ is the adjoint map of $q$. For example, each element $y \in M$ gives rise to its open quotient map $\map{\open{y}}{M}{\downset{y}} = (x \mapsto y \wedge x)$, which is associated with the open sublocale 
\[
\open{y}_*(\downset{y})
= \setof{y \to x}{x \in M}
\equiv y \to M.
\] 
Recall that for any frame surjection $\map{q}{M}{L}$, the associated sublocale $S = q_*(L)$ is contained in the open sublocale $y \to M$ iff $q(y) = \top$ (\cite[1.3.1]{PicadoPultr:2012}). Recall also that any sublocale in a fit frame, and therefore in a regular frame, is the intersection of the open sublocales which contain it (\cite[1.3.2]{PicadoPultr:2012}), i.e., $S = \cap \setof{y \to M}{q(y) = \top}$. Finally, recall that a frame surjection is said to be \emph{dense} if $q(x) = \bot$ implies $x = \bot$, i.e., if $\bot \in q_* (L)$. 

A cozero element of a frame $L$ is one of the form $g(\mbb{R} \smallsetminus \{0\})$, $g \in \mcal{R} L$. (See \cite[XIV 6.2.3]{PicadoPultr:2012} for a characterization.) By replacing $g$ by $\left|g\right|$, we may assume a cozero element to be of the form $g(0, \infty)$ for $g \in \mcal{R}^+ L$. In a pointed frame $(L, *)$, a cozero element which contains $*$ is of the form $g(0, \infty)$ for some $g \in \mcal{R}_0^+ L$, and a cozero element which omits star is of the form $\con g \equiv g(-\infty, 1)$ for some $g \in \mcal{R}_0^+ L$. By replacing $g$ by $\ol{g}$ in the latter cases we may assume that $g = \ol{g}$. For a subtrunc $G \subseteq \mcal{R}_0 L$ we let $\coz \ol{G} 
\equiv \setof{\coz g}{g \in \ol{G}}$ and $\con \ol{G} \equiv \setof{\con g}{g \in \ol{G}}$. 

\begin{definition*}[suitable compactification]	
	A \emph{compactification} is a dense surjective pointed frame homomorphism $\map{q}{M}{L}$ with compact domain $M$. The compactification is said to be \emph{suitable} if the corresponding sublocale $q_*(L)$ is the intersection of the cozero sublocales containing it. For compactifications $\map{q}{M}{L}$ and $\map{r}{N}{Q}$, a \emph{homomorphism of compactifications} is a pair $(l, m)$ of pointed frame homomorphisms such that $m \circ q = r \circ l$.  
	\[
	\begin{tikzcd}
	M \arrow{r}{q} \arrow{d}[swap]{l}
	& L \arrow{d}{m}\\ 
	N \arrow{r}[swap]{r}& Q 
	\end{tikzcd}	
	\]	
	We write $\map{(l, m)}{q}{r}$. We denote the category of compactifications with their homomorphisms by $\SC$. 
\end{definition*}

\subsection{Representing $G$ as a subtrunc of $\mcal{R}_0 L$\label{Subsec:MaddenRep}}

We denote by $\map{p} {\mcal{O}_* \exR} {\mcal{O}_* \mbb{R}}$ the pointed frame homomorphism $(U \mapsto U \cap (-\infty, \infty))$ of the inclusion $\mbb{R} \to \exR$. 

\begin{lemma}\label{Lem:4}
	Let $\map{q}{M}{L}$ be a dense pointed frame surjection. Then a homomorphism $h'$ \enquote*{drops} to a homomorphism $h$ such that $h \circ p = q \circ h'$ iff $q \circ h'(-\infty, \infty) = \top$.
	\[
	\begin{tikzcd}
	\mcal{O}_* \exR \arrow{r}{h'} \arrow{d}[swap]{p} & M \arrow{d}{q}\\
	\mcal{O}_*\mbb{R} \arrow{r}[swap]{h} & L
	\end{tikzcd}
	\]
\end{lemma}

\begin{proof}
	If $h'$ satisfies $h \circ p = q \circ h'$ for some $h \in \mcal{R}_0 L$ then $q \circ h'(-\infty, \infty) = h \circ p (-\infty, \infty) = h(\top) = \top$. On the other hand, suppose that $q \circ h'(-\infty, \infty) = \top$. Then for $U_i \in \mcal{O}_*\exR$, if $p(U_1) = p(U_2)$ then $U_1 \cap (-\infty, \infty) = U_2 \cap (-\infty, \infty)$, hence
	\begin{align*}
	q \circ h'(U_1)
	&= q \circ h'(U_1) \wedge \top
	=  q \circ h'(U_1) \wedge q \circ h'(-\infty, \infty)
	= q \circ h'(U_1 \cap (-\infty, \infty))\\
	& = q \circ h'(U_2 \cap (-\infty, \infty))
	= q \circ h'(U_2) \wedge q \circ h'(-\infty, \infty)
	= q \circ h'(U_2).
	\end{align*}
	This fact allows us to define $h(U) \equiv q \circ h'(U')$ for some (any) $U' \in \mcal{O}_*\exR$ such that $p(U') = U$. This yields a homomorphism $h$ which makes the diagram commute. 
\end{proof}

Notice that if $M$ is compact then the only elements $h'$ which can possibly drop as in Lemma \ref{Lem:4} are those of $\mcal{D}_0 M$. 

\begin{definition*}[$\mcal{E}_0 q$]
	For a dense pointed frame surjection $\map{q}{M}{L}$, we let
	\[
	\mcal{E}_0 q
	\equiv \setof{h \in \mcal{R}_0 L}{\exists h' \in \mcal{D}_0 M\ (q \circ h' = h \circ p)}.
	\]
	When referring to an element $g \in \mcal{E}_0 q$, we denote by $g'$ that element of $\mcal{D}_0 M$ such that $g \circ p = q \circ g'$. 
\end{definition*}

Like $\mcal{D}_0 X$ and $\mcal{D}_0 M$, $\mcal{E}_0q$ is closed under scalar multiplication, the lattice operations, and truncation. However, it is not generally a trunc, as shown by  Example \ref{Ex:1}. Nevertheless, a subset $G \subseteq \mcal{E}_0 q$ may be closed under all of the trunc operations, in which case we refer to $G$ as a \emph{trunc in $\mcal{E}_0 q$.}

\begin{definition*}[trunc snugly embedded in $\mcal{E}_0 q$]
	Let $\map{q}{M}{L}$ be a dense pointed frame surjection. We shall say that a trunc $G$ in $\mcal{E}_0 q$ is \emph{snugly embedded} if $\coz \ol{G} \cup \con \ol{G}$ join-generates $M$, and if the sublocale $S \equiv q_* L$ is the intersection of the open sublocales of the form $g'(-\infty, \infty) \to M$, $g \in \ol{G}$.   
\end{definition*}

\begin{theorem}\label{Thm:3} 
	Let $G$ be a trunc.
	\begin{enumerate}
		\item 
		There is a suitable compactification $q$, a snugly embedded trunc $\widehat{G} \subseteq \mcal{E}_0 q$, and a trunc isomorphism $\map{\mu_G}{G}{\widehat{G}}$. The trunc $\widehat{G} \subseteq \mcal{R}_0 L$, or the isomorphism $\mu_G$, is referred to as the \emph{Madden representation of the trunc $G$}. Up to isomorphism, it is unique with respect to its properties.
		
		\item 
		This representation is functorial. For any homomorphism $\map{\theta}{G}{H}$, where $H$ is a trunc with suitable compactification $\map{r}{N}{Q}$ and isomorphism $\map{\mu_H}{H}{\mcal{E}_0 r}$, there exists a unique $\SC$-morphism $\map{(l, m)}{q}{r}$ such that $\mcal{E}_0(l, m) \circ \mu_G = \mu_H \circ \theta$.
		\begin{figure}[htb]
			\begin{tikzcd}
			G \arrow{r}{\mu_G} \arrow{dd}[swap]{\theta}
			& \mcal{E}_0q \arrow{dd}{\mcal{E}_0(k,l)}
			& M \arrow{r}{q} \arrow{dd}[swap]{l}
			&L \arrow {dd}{m}&\\ 
			&&&&\mcal{O}_* \mbb{R} \arrow{ul}[swap]{\hat{g}} \arrow {dl} {\theta (g)}\\
			H \arrow{r}{\mu_H} &\mcal{E}_0r
			& N \arrow{r}[swap]{r}
			& Q& 
			\end{tikzcd}	
		\end{figure}
	\end{enumerate}
\end{theorem} 

\begin{proof}
	Let $G$ be a given trunc with Yosida representation $\map{\nu_G}{G} {\widetilde{G}} \subseteq \mcal{D}_0 X$ per Theorem \ref{Thm:16}. Let $M \equiv \mcal{O}_* X$ represent the Yosida frame of $G$, and abbreviate $\map{\tilde{g}^{-1}} {\mcal{O}_*\exR} {M}$ to $g'$ for  each $g \in G$. Let $\map{q}{M}{L}$ be the intersection of the open quotients $M \to \downset{g'(-\infty, \infty)} = (x \mapsto x \wedge g'(-\infty, \infty))$, $g \in G$. Since each $\tilde{g}$ lies in $\mcal{D}_0 X$, the corresponding pointed frame map makes $g'(-\infty, \infty)$ a dense element of $M$ and the corresponding open quotient map $M \to \downset{g'(-\infty, \infty)}$ a dense surjection. Consequently $q$ is a dense surjection, and, by Lemma \ref{Lem:4}, each map $g'$, $g \in G$, drops to a unique pointed frame map $\map{\hat{g}}{\mcal{O}_* \mbb{R}}{L}$ such that $q \circ g' = \hat{g} \circ p$. Finally, the trunc homomorphism $\map{\mu_G}{G}{\mcal{R}_0 L}$ thus defined is one-one, for if $0 \neq g \in G$ then $g'([-\infty, 0) \cup (0, \infty]) > \bot$ in $M$, hence 
	\[
	\bot 
	< q \circ g'([-\infty, 0) \cup (0, \infty]) 
	= \hat{g} \circ p([-\infty, 0) \cup (0, \infty]) 
	= \hat{g}((-\infty, 0) \cup (0, \infty)),
	\] 
	so $\hat{g} \neq 0$.  
	
	Now consider a trunc homomorphism $\map{\theta}{G}{H}$, where $H$ is a trunc with Yosida representation $\map{\nu_H}{H}{\widetilde{H}} \subseteq \mcal{D}_0 Y$. Let $\map{k}{Y}{X}$ be the continuous function provided by Theorem \ref{Thm:16} such that $\widetilde{\theta(g)} = \nu_H \circ \theta(g) = \nu_G(g)\circ k = \tilde{g}\circ k$ for all $g \in G$. Let $N \equiv \mcal{O}_* Y$ be the Yosida frame of $H$, and abbreviate $\map{\tilde{h}^{-1}} {\mcal{O}_*\exR}{N}$ to $h'$ for each $h \in H$. Let $\map{l}{M}{N}$ be the pointed frame homomorphism corresponding to $k$, so that if we denote the frame maps of $\tilde{g}$ and $\widetilde{\theta(g)}$ by $g'$ and $\theta(g)'$, respectively, we get that $\theta(g)' = l \circ g'$. Finally, let $\map{r}{N}{Q}$ be the intersection of the open quotients $N \to \downset{h'(-\infty, \infty)}$, $h \in H$. As before, $r$ is a dense surjection, so each $h'$, $h \in H$, drops to a unique $\map{\hat{h}}{\mcal{O}_*\mbb{R}}{Q}$ such that $r \circ h' = \hat{h} \circ p$. 
	\begin{figure}[htb]
	\begin{tikzcd}
	M \arrow{rrr}{q} \arrow{dd}[swap]{l}
	&&& L \arrow{dd}{m}\\
	& \mcal{O}_* \exR \arrow{ul}[swap]{g'} \arrow{r}{p} \arrow{dl}{\theta(g)'}
	& \mcal{O}_* \mbb{R} \arrow{ur}{\hat{g}} \arrow{dr}[swap] {\widehat{\theta(g)}}&\\
	N \arrow{rrr}[swap]{r}
	&&&Q
	\end{tikzcd}
	\end{figure}
	Now $q$ is associated with the finest frame congruence which identifies each $g'(-\infty, \infty)$ with $\top$, $g \in G$, and likewise $r$ is associated with the finest frame congruence which identifies each $h'(-\infty, \infty)$ with $\top$, $h \in H$. In light of the fact that $\theta(g)' = l \circ g'$ for each $g \in G$, it follows that $r \circ l (x_1) = r \circ l(x_2)$ for all $x_i \in M$ such that $q(x_1) = q(x_2)$. From this, in turn, follows the existence of a unique frame homomorphism $m$ such that $m \circ q = r \circ l$. We have
	\begin{align*}
	m \circ \hat{g} \circ p
	&= m \circ q \circ g'
	= r \circ l \circ g'
	= r \circ \theta(g)'
	= \widehat{\theta(g)}\circ p.
	\end{align*}
	But $p$ is surjective and hence epi, from which $m \circ \hat{g} = \widehat{\theta(g)}$ follows.
\end{proof}

Theorem \ref{Thm:3} is a hybrid of the Yosida Representation Theorem \ref{Thm:16} and the Madden Representation Theorem (\cite[5.1.1, 5.3.1]{Ball:2014.2}). It is sharper than these two, however, in the sense that its target $\mcal{E}_0 q$ is a smaller object than either $\mcal{D}_0 X$, the target of the Yosida Theorem, or $\mcal{R}_0 L$, the target of the Madden Theorem. In fact, $\mcal{E}_0 q$ can be viewed as the \enquote{intersection} of $\mcal{D}_0 X$ and $\mcal{R}_0 L$ in a sense made precise in Lemma \ref{Lem:4}. Nevertheless, Theorem \ref{Thm:3} shares with the Yosida Theorem the disadvantage that its target is not generally a trunc.  We shall return to the topic of $\mcal{E}_0 q$ in Subsection \ref{Subsec:E0q} and in Section \ref{Sec:E0q}.

\begin{example}\label{Ex:1}
	Let $X$ designate the compact pointed space $([0,1], 1)$, and let $\tilde{g}_0 \in \mcal{D}_0 X$ be the function 
	\[
	\tilde{g}_0(x)
	\equiv 
	\begin{cases}
	\infty      &\text{if $x = 0$}\\
	1/x - 1     &\text{if $x > 0$}
	\end{cases}.
	\]
	Let $\widetilde{G} \equiv \setof{\tilde{f} + r\tilde{g}_0}{\tilde{f} \in \mcal{C}_0 X,\ r \in \mbb{R}}$, a subtrunc of $\mcal{D}_0 X$ presented in its Yosida representation.  Let $M \equiv \mcal{O}_* X$ and $L \equiv \mcal{O}_* (0, 1]$; the insertion $(0,1] \to X$ gives rise to the frame map $\map{q}{M}{L} = (U \mapsto U \cap (0,1])$, and, in fact, the map $\widetilde{G} \to \mcal{R}_0 L = (\tilde{g} \mapsto \tilde{g}^{-1}) \equiv g$ is the Madden representation of $\widetilde{G}$ and $q$ is the compactification of Theorem \ref{Thm:3}.
	
	Now $G \subseteq \mcal{E}_0q$ by construction, but reasoning similar to that used for $\tilde{g}_0$ leads to the conclusion that $\mcal{E}_0q$ also contains the frame homomorphism $f_0$ dual to the continuous function
	\[
	\tilde{f}_0(x)
	\equiv 
	\begin{cases}
	\infty      &\text{if $x = 0$}\\
	1/x + \sin (1/x) - 1 - \sin 1     &\text{if $x > 0$}
	\end{cases}.	
	\]
	The point is that $h \equiv f_0 - g_0 \in \mcal{R}_0 L \smallsetminus \mcal{E}_0q$, for $\tilde{h} \equiv \tilde{f}_0 - \tilde{g}_0 = sin (1/x) - \sin 1 \notin \mcal{D}_0 X$. 
\end{example} 

\subsection{Trunc notation}\label{Subsec:Notation}
Henceforth the symbol $G$ stands for a trunc with Yosida space $\mcal{Y}_* G \equiv (X,*)$, Yosida frame $M = \mcal{O}_* X$, and suitable compactification $\map{q}{M}{L}$. When $G$ is being regarded as an abstract $\mbf{T}$-object or as a subtrunc of $\mcal{R}_0 L$, we shall denote it by the symbol $G$ and denote its members by unadorned lower case letters such as $f$ and $g$. In particular, this convention  applies when $G$ is regarded as a trunc in $\mcal{E}_0 q$, in which case we use $f'$ and $g'$ to denote the members of $\mcal{D}_0 M$ which correspond to $f$ and $g$ as in Lemma \ref{Lem:4}. On the other hand, on those occasions when it is advantageous to regard $G$ as a trunc in $\mcal{D}_0 X$ (Subsection \ref{Subsec:YosRepTruncs}), we shall denote it by the symbol $\widetilde{G}$ and denote its members by symbols like $\tilde{f}$ and $\tilde{g}$. 

\section{Uniform and Pointwise convergence\label{Sec:PtwiseCon}}

The two most classical convergences have elegant formulations in the language of truncs. We introduce them here, for both are important for our purposes. We begin by defining uniform convergence in archimedean truncs. 
 
\subsection{Uniform convergence}

\begin{proposition}\label{Prop:16}
	The following are equivalent for a sequence $\{g_n\} \subseteq \ol{G}$.
	\begin{enumerate}
		\item 
		$\forall k\ \exists m\ \forall n \geq m\ (kg_n = \ol{kg_n})$
		
		\item 
		$\forall \varepsilon > 0\ \exists m\ \forall n \geq m\ \forall x \in X\ (\tilde{g}_n(x) < \varepsilon)$.
			
		\item 
		$\forall \varepsilon > 0\ \exists m\ \forall n \geq m\ (g_n(-\infty, \varepsilon) = \top)$.
	\end{enumerate}
\end{proposition}

\begin{proof}
	The equivalence of (1) with (2) follows directly from the fact that $\ol{\tilde{g}}(x) = \tilde{g}(x) \wedge 1$ for all $x \in X$ and $g \in G^+$. The equivalence of (1) and (3) follows directly from the fact that $\ol{g}(-\infty, \varepsilon) = g(-\infty, \varepsilon)$ for $\varepsilon < 1$ by Lemma \ref{Lem:8}.
\end{proof}

\begin{definition*}[uniform convergence]
	A sequence $\{g_n\} \subseteq \ol{G}$ is said to \emph{converge uniformly to $0$}, written $g_n \to 0$, if it satisfies the conditions of Proposition \ref{Prop:16}. A sequence $\{g_n\} \subseteq G$ converges to an element $g_0$ if $\left|g_n - g_0\right| \to 0$; we write $g_n \to g_0$. A sequence $\{g_n\} \subseteq G^+$ is said to be \emph{uniformly Cauchy} if 
	\[
	\forall k\ \exists m\ \forall i,j \geq m\ \big(k\left|g_i - g_j\right| = \ol{k\left|g_i - g_j\right|}\big).
	\] 
	A trunc $G$ is said to be \emph{uniformly complete} if every uniformly Cauchy sequence $\{g_n\} \subseteq G$ converges uniformly to a limit in $G$.
\end{definition*}

We record the well known classical theory of uniform convergence, including the Stone-Weierstrass Theorem.

\begin{proposition}
	Uniform convergence interacts nicely with the trunc structure.
	\begin{enumerate}
		\item 
		All of the trunc operations are uniformly continuous.  That is, if $f_n \to f_0$ and $g_n \to g_0$ then $(f_n \oplus g_n) \to f_0 \oplus g_0$, where $\oplus$ may be taken to be $+$, $-$, $\vee$, or $\wedge$. Likewise $f_n \to f_0$ implies $\ol{f}_n \to \ol{f}_0$.
		
		\item 
		Any trunc homomorphisms $\map{\theta}{G}{H}$ is uniformly continuous, i.e., $g_n \to g_0$ in $G$ implies $\theta(g_n) \to \theta(g_0)$ in $H$.  	
		
		\item 
		$G$ is uniformly complete iff $\widetilde{G}^* = \mcal{C}_0 X$. 
	\end{enumerate}
\end{proposition}

\subsection{Directional pointwise convergence}\label{Subsec:DirPtwiseCon}
Directional pointwise convergence will play a modest but important role in our development. (Unrestricted pointwise convergence plays a central role in the theory of the pointfree Baire functions, the subject of the more recent article \cite{Ball:2018}.)  Directional pointwise convergence is based on the notion of pointfree pointwise suprema, the definition of which presumes that $G$ is a identified with its Madden representation as a subtrunc of $\mcal{R}_0 L$ for a pointed frame $L$.

\begin{definition*}[pointwise suprema and infima] \label{Def:1}
	An element $b \in G$ is said to be the \emph{pointwise supremum (infimum)} of a subset $A \subseteq G$, written $\bigvee^\bullet A = b$ \big($\bigwedge^\bullet A = b$\big), if for all $r \in \mbb{R}$,
	\[
	\bigvee_A a(r, \infty) = b(r, \infty)
	\quad \bigg(\bigvee_A a (-\infty, r) = b(-\infty, r)\bigg). 
	\]
\end{definition*}

Even though pointwise suprema and infima are defined in $\mcal{R}_0 L$, they are characterized in Proposition \ref{Prop:21} in terms independent of any representation. They are precisely those joins and meets which remain valid whenever $G$ is embedded in a larger trunc. To coin a phrase, pointwise joins and meets are those which are \emph{context free}. 

\begin{proposition}[cf.\ \cite{BallHagerWW:2015}, 4.2.5]\label{Prop:21}
	The following are equivalent for an element $b$ and subset $A$ of a trunc $G$.
\begin{enumerate}
	\item 
	$\bigvee^\bullet A = b$.
	
	\item 
	$\bigvee \theta(A) = \theta(b)$ for every truncation homomorphism $\theta$ out of $G$.
\end{enumerate}  
\end{proposition}

The article \cite{BallHagerWW:2015} is a lengthy treatment of pointwise suprema and infima, carried out in $\mbf{W}$ but yielding results which carry over to $\mbf{T}$.   

\begin{definition*}[directional pointwise convergence]
	A sequence $\{g_n\}$ in a trunc $G$ \emph{converges pointwise downwards (upwards) to an element $g_0 \in G$}, written $g_n\searrow g_0$ ($g_n\nearrow g_0$), provided that $g_{n + 1} \geq g_n$ ($g_{n+1} \leq g_n$) for all $n$, and $\bigwedge^\bullet g_n = g_0$ ($\bigvee^\bullet g_n = g_0$), i.e., for all $r \in \mbb{R}$,
	\[
		\bigvee_n g_n(-\infty, r) = g_0(-\infty,r)\quad
		\bigg(\bigvee_n g_n(r,\infty) = g_0(r,\infty)\bigg). 
	\]
\end{definition*}

Directional pointwise convergence has a number of nice propertie relevant to our development, and we list the main ones here.

\begin{proposition}[\cite{Ball:2018}, 5.2, 5.3]\label{Prop:3}
	Let $\{f_n\}$ and $\{g_n\}$ be sequences in $\mcal{R}_0 L$, and let $f_0$ and $g_0$ be elements of $\mcal{R}_0 L$.
	\begin{enumerate}
		\item 
		$\dot{g} \searrow g$, where $\dot{g}$ designates the constant sequence, i.e., $\dot{g}_n = g$ for all $n$.
		
		\item 
		$g_n \searrow g_0$ iff $(-g_n) \nearrow (-g_0)$.
		
		\item
		If $\{g_n\} \subseteq \mbb{R}_0^+ L$ and $g_n \searrow g_0$ then $\ol{g_n} \searrow \ol{g_0}$.
		
		\item 
		If $f_n \searrow f_0$ and $g_n \searrow g_0$ then $(f_n	\oplus g_n) \searrow (f_0 \oplus g_0)$, where $\oplus$ stands for one of the operations $+$, $\vee$, or $\wedge$. 
		
		\item 
		If $g_n \searrow g_0$ and $0 \leq r \in \mbb{R}$ then $(rg_n) \searrow rg_0$. 
		
		\item 
		If $g_n\searrow g_0$ and $g_n \searrow f_0$ then $g_0 = f_0$.
		
		\item 
		Directional pointwise convergence is continuous. That is, for any $\mbf{T}$-homomorphism $\tau \colon \mcal{R}_0 M \to \mcal{R}_0 N$, and for any sequence $\{g_n\}$ and element $g_0$ of $\mcal{R}_0 M$,
		\[
		g_n \searrow g_0 \text{ in $\mcal{R}_0 M$}
		\implies \tau(g_n) \searrow \tau(g_0) \text{ in $\mcal{R}_0 N$}.
		\]
	\end{enumerate}
	The dual statements for pointwise upwards convergence all hold as well.
\end{proposition}

For a compact frame $L$, directional pointwise convergence in $\mcal{R}_0 L$ reduces to uniform convergence. This classical theorem of Dini has a concise demonstration in a pointfree context. 

\begin{theorem}[\cite{Dini:1878}]\label{Thm:18}
	Let $\{g_n\}$ be a nonincreasing sequence of positive elements in $\mcal{R}_0 L$ for a compact frame $L$. Then 
	\[
	g_n \searrow 0 
	\implies g_n \to 0.
	\]
\end{theorem}

\begin{proof}
	To say that $g_n \searrow 0$ is to say that $\bigvee_n g_n(-\infty, \varepsilon) = \top$ for all $\varepsilon > 0$. But since $L$ is compact, that means that for all $\varepsilon > 0$ there is an index $m$ such that $g_n(-\infty, \varepsilon) = \top$ for all $n \geq m$, which is to say that $g_n \to 0$ by Proposition \ref{Prop:16}.
\end{proof}

\section{Truncation sequences, good sequences \label{Sec:TruncSeq}}

\subsection{Truncation sequences \label{Subsec:TruncSeq}}

In the context of $\mbf{W}$-objects, the notion of a truncation sequence was mentioned briefly in Section 5.2 of \cite{BallHagerWW:2015}. The closely related notion of \enquote*{expanding sequence} plays an important role in Hager's penetrating analysis of $*$-maximal $\mbf{W}$-objects in \cite[S.\ 10]{Hager:2013}. Here we generalize the idea of a truncation sequence to truncs, and relate it to Mundici's good sequences in Proposition \ref{Prop:27}. The major result is the characterization of the elements of $\mcal{E}_0 q$ in terms of truncation sequences in Proposition \ref{Prop:1}.

\begin{definition*}[truncation sequence]
	The \emph{n\textsuperscript{th} truncation} of an element $g \in G^+$ is $n\ol{g/n}$. The \emph{truncation sequence} of $g$ is $\{n\ol{g/n}\}$. We often abbreviate $n\ol{g/n}$ to $g \wedge n$, trusting the reader to keep in mind that the symbol $n$ is formal and not a reference to a member of the trunc.
\end{definition*}

\begin{lemma}\label{Lem:27}
	Let $\{g \wedge n\}$ be the truncation sequence of an element $g \in G^+$. If $G$ is represented as a trunc in $\mcal{D}_0 X$ for a pointed space $X$ then $\widetilde{g \wedge n}(x) = \tilde{g}(x) \wedge n$ for all $x \in X$ and all $n$. If $G$ is represented as a subtrunc of $\mcal{R}_0 L$ for a pointed frame $L$ then
	\[
		(g \wedge n)(-\infty, r)
		= \begin{cases}
		\top        &\text{if $r > n$},\\
		g(-\infty, r)  &\text{if $r \leq n$},
		\end{cases},\qquad r \in \mbb{R},\ n \in \mbb{N}.
	\] 
\end{lemma}

\begin{proof}
	By Lemma \ref{Lem:8} we get that for any $r \in \mbb{R}$,
	\begin{align*}
		(g \wedge n)(-\infty, r)
		&= n\ol{g/n}(-\infty,r)
		= \ol{g/n}(-\infty, r/n)
		= \begin{cases}
		\top    & \text{if $r/n > 1$}\\
		(g/n)(-\infty, r/n)   & \text{if $r/n \leq 1$}
		\end{cases}\\
		&= \begin{cases}
		\top    & \text{if $r > n$}\\
		g(-\infty, r)   & \text{if $r \leq n$}
		\end{cases}.\qedhere
	\end{align*}
\end{proof}

\begin{corollary}\label{Cor:1}
	For all $g \in G^+$, $g	= \bigvee_n^\bullet (g \wedge n)$.
\end{corollary}

\begin{proof}
	To say that $g = \bigvee_n^\bullet (g \wedge n)$ is to say that $\bigvee_n n\ol{g/n}(-\infty, r) = g(-\infty, r)$ for all $r \in \mbb{R}$. But this is clear from Lemma \ref{Lem:27}.
\end{proof}

\begin{proposition}[cf.\ \cite{BallHagerWW:2015}, 5.2.2]\label{Prop:26}
	A sequence $\{g_n\} \subseteq \mcal{R}_0^+ L$ is the truncation sequence of some element $g \in \mcal{R}_0^+ L$ iff 
	\begin{enumerate}
		\item 
		$g_n = g_{n + 1} \wedge n$ for all $n$, and
		
		\item 
		$\bigvee_n^\bullet g_n = \top$, i.e., $\bigvee_n g_n(-\infty, n) = \top$ in $L$.
	\end{enumerate}
	In this circumstance $g_n \nearrow g$.  
\end{proposition}

\begin{proof}
	Suppose $g \in G^+$ and $g_n =g \wedge n = n \ol{g/n}$ for all $n$. Then $\{g_n\}$ satisfies (1) by Lemma \ref{Lem:27} and (2) by virtue of the fact that $\bigvee_n g_n(-\infty, n) = \bigvee_n g(-\infty, n) = \top$. 
	
	Now suppose $\{g_n\}$ is a sequence in $G^+$ which satisfies (1) and (2). Define the function $\map{g}{\setof{(-\infty, r)}{r \in \mbb{R}}}{L}$ by putting $g(-\infty, r) \equiv g_n(-\infty, r)$ for some (any) $n > r$. The function is well defined because the sequence satisfies (1), and it follows from Lemma \ref{Lem:27} that $g_n = g_m \wedge n$ for all $m \geq n$. According to \cite[3.1.2]{BallHager:1991}, the function has a unique extension to a member of $\mcal{R}_0 L$ iff it has the following properties.
	\begin{enumerate}
		\item[(E1)]
		$r < s$ implies $g(-\infty, r) \prec g(-\infty, s)$.
		
		\item[(E2)] 
		$\bigvee_{s < r} g(-\infty, s) = g(-\infty, r)$.
		
		\item [(E3)]
		$\bigvee_{r} g(-\infty, r) = \bigvee_r g(-\infty, r)^* = \top$.  
	\end{enumerate}
	Properties (E1) and (E2) hold because they hold for each $g_n$, and $g(-\infty, r) = g_n(-\infty, r)$ for any $n > r$. That $\bigvee_r g(-\infty, r) = \top$ follows from (2), while $\bigvee_r g(-\infty, r)^* = \top$ because $g(-\infty, 0) = g_1(-\infty, 0) = \bot$. Finally, Lemma \ref{Lem:27} makes it clear that the sequence of truncations of $g$ is $\{g_n\}$.   
\end{proof}

Proposition \ref{Prop:6} is the analog of Proposition \ref{Prop:26} for subtruncs of $\mcal{D}_0 M$, $M$ a compact pointed frame. It plays an important role in the proof of our crucial Proposition \ref{Prop:1}. Its proof requires a lemma analogous to Lemma 3.1.2 of \cite{BallHager:1991}. Consistent with previous notation, we denote members of $\mcal{D}_0 M$ by symbols like $f'$ and $g'$.

\begin{lemma}\label{Lem:5}
	A function $\map{g'}{\setof{[-\infty, r)}{r \in \mbb{R}}}{M}$, $M$ a compact pointed frame, can be extended to a member of $\mcal{D}_0 M$ iff it satisfies the following conditions for all $r,s \in \mbb{R}$. The extension is unique when it exists.
	\begin{enumerate}
		\item 
		$* \in g'[-\infty, r)$ iff $r > 0$. 
		
		\item 
		$r < s$ implies $g'[-\infty, r) \prec g'[-\infty, s)$.

		\item 
		$\bigvee_{s < r} f[-\infty, s) = f[-\infty, r)$.
		
		\item 
		$\bigvee_r g'[-\infty, r)$ and $\bigvee_r g'[-\infty, r)^*$ are dense elements of $M$.
	\end{enumerate}
\end{lemma}

\begin{proof}
	First extend $g'$ by defining $g'(r, \infty] \equiv \bigvee_{r < q} g'[-\infty, q)^*$ for $r \in \mbb{Q}$, then put $g'[-\infty, \infty] \equiv \top$, and let $g'(r,s) \equiv g'[-\infty, s) \wedge g'(r, \infty]$ for $r,s \in \mbb{Q}$. With these modifications, the proof now closely follows the argument used to prove \cite[3.1.2]{BallHager:1991}.     
\end{proof}

\begin{proposition}\label{Prop:6}
	A sequence $\{g_n'\} \subseteq \mcal{D}_0^+ M$, $M$ a compact pointed frame, is the truncation sequence of some element $g' \in \mcal{D}_0^+ M$ iff
	\begin{enumerate}
		\item 
		$g'_n = g_{n + 1}' \wedge n$ for all $n$, and
		
		\item 
		$\bigvee_n g_n'[-\infty, n)$ is dense in $M$. 
	\end{enumerate}
\end{proposition}

\begin{proof}
	Argue as in the proof of Proposition \ref{Prop:26}, using Lemma \ref{Lem:5}.
\end{proof}

\subsection{Good sequences}\label{Subsec:GoodSeq}

In this subsection we take up the connection between the truncation sequences of Subsection \ref{Subsec:TruncSeq} and a close analog of Mundici's good sequences, a crucial construct used in his celebrated result \cite{Mundici:1986} linking $MV$-algebras and unital $\ell$-groups. We emphasize the similarity by expropriating Mundici's terminology.

\begin{definition*}[good sequence]
	A sequence $\{f_n\} \subseteq \ol{G}$ is called a \emph{good sequence} if $f_n = \ol{f_n + f_{n + 1}}$ for all $n$, and if $f_n \searrow 0$. 
\end{definition*}

We record a couple of fundamental trunc identities.

\begin{lemma}\label{Lem:28}
	The following hold for any $g \in G^+$ and $n \in \mbb{N}$.
	\begin{enumerate}
		\item
		$g \wedge n + g \ominus n = g$.
		
		\item 
		$g \wedge n + \ol{g \ominus n} = g \wedge (n + 1)$. 
	\end{enumerate}
\end{lemma}

\begin{proof}
	To establish (1), observe that 
	\[
	g \wedge n + g \ominus n
	= n\ol{g/n} + n((g/n) \ominus 1)
	= n\ol{g/n} + n ((g/n) - \ol{g/n})
	=g.
	\]
	To establish (2) note that 
	\begin{align*}
	g \wedge (n + 1)
	&= g - g \ominus (n + 1)
	&&\text{by (1),}\\
	&= g - g \ominus n \ominus 1
	&&\text{by \cite[3.3.6]{Ball:2014.1},}\\
	&= g + (\ol{g \ominus n} - g \ominus n)
	&&\text{by definition of $\ominus$,}\\
	&= (g \wedge n + g \ominus n) + (\ol{g \ominus n} - g \ominus n)
	&&\text{by (1),}\\
	&= g \wedge n + \ol{g \ominus n}.\qedhere
	\end{align*}
\end{proof}

\begin{lemma}\label{Lem:2}
	The following hold for any $g \in G^+$ and $m \in \mbb{N}$.
	\begin{enumerate}
		\item 
		The sequence $\{\ol{g \ominus (n - 1)}\}$ is a good sequence .
		
		\item 
		$g = \sum_{1 \leq n \leq m} \ol{g \ominus (n - 1)} + g \ominus m$.
		
		\item 
		$g \wedge m = \sum_{1 \leq n \leq m }\ol{g \ominus(n - 1)}$ 
		
		\item 
		If $\{f_n\}$ is a good sequence and $g \equiv \sum_{1 \leq n \leq m}f_n$ then $f_n = \ol{g \ominus(n - 1)}$ for all $n \leq m$.
	\end{enumerate}
\end{lemma}

\begin{proof}
	(1) Let $f_n \equiv \ol{g \ominus (n - 1)}$ and $f \equiv (g \ominus (n - 1))/2$. Then, using parts (5) and (9) of \cite[3.3.1]{Ball:2014.1},
	\[
	\ol{f_n + f_{n+1}} 
	= \ol{\ol{g \ominus (n - 1)} + \ol{g \ominus n}}
	= \ol{\ol{2f} + \ol{2f \ominus 1}}
	= \ol{2 \ol{f}}
	= \ol{2f}
	= \ol{g \ominus (n - 1)}
	= f_n.
	\]
	To show that $f_n \searrow 0$ we must show that $\bigwedge_n^\bullet f_n = 0$, which is to say that for all $\varepsilon > 0$, 
	\begin{align*}
	\bigvee_n f_n(-\infty, \varepsilon)
	&= \bigvee_n \ol{g \ominus (n - 1)}(-\infty, \varepsilon)
	= \bigvee_n g \ominus (n - 1)(-\infty, \varepsilon)\\
	&= \bigvee_n g(-\infty, n - 1 +  \varepsilon)
	= \top.
	\end{align*}
	The second equality holds for $\varepsilon < 1$ by one part of Lemma \ref{Lem:8}, and the third equality is a minor variant of another part of the same lemma. 
	
	(2) can be proven by induction; the basis, i.e., the $m = 1$ step, is the assertion that $g = \ol{g} + g \ominus 1$. To establish the induction step, use \cite[3.3.6]{Ball:2014.1} to get
	\begin{align*}
	g
	&= \smashoperator[lr]{\sum_{1 \leq n \leq m}} \ol{g \ominus (n - 1)} + g \ominus m
	= \smashoperator[lr]{\sum_{1 \leq n \leq m}} \ol{g \ominus (n - 1)} + \ol{g \ominus m} + g \ominus m \ominus 1\\
	&= \smashoperator[lr]{\sum_{1 \leq n \leq m}} \ol{g \ominus (n - 1)} + \ol{g \ominus m} + g \ominus (m + 1)
	= \smashoperator[lr]{\sum_{1 \leq n \leq m + 1}} \ol{g \ominus (n - 1)} + g \ominus (m + 1).
	\end{align*}
	(3) follows from (2) together with Lemma \ref{Lem:28}.
	
	(4) Suppose $g = f_1 + f_2 + \ldots + f_m$. Then by \cite[3.3.1]{Ball:2014.1}(11),
	\begin{align*}
	\ol{g} 
	&= \ol{f_1 + f_2 + \ldots + f_m}
	= \ol{f_1 + f_2 + \ldots \ol{(f_{m-1} + f_m)}}
	= \ol{f_1 + f_2 + \ldots f_{m-1}}\\
	&= \ol{f_1 + f_2 + \ldots \ol{(f_{m - 2} + f_{m - 1})}}
	= \ol{f_1 + f_2 + \ldots f_{m-2}}
	= \ldots = \ol{f_1}
	= f_1.
	\end{align*}
	Therefore $g \ominus 1 = g - \ol{g} = (f_1 + f_2 + \ldots + f_m) - f_1 = f_2 + \ldots + f_m$, and a repetition of the preceding argument shows that $\ol{g \ominus 1} = f_2$. Continue.
\end{proof}

The terminology in Proposition \ref{Prop:27} requires a little clarification. We shall say that a sequence $\{g_n\} \subseteq G^+$ is a \emph{truncation sequence} if it is the truncation sequence of some element $g_0 \in \mcal{R}_0 L$, i.e., if it has properties (1) and (2) of Proposition \ref{Prop:26}. If $\{g_n\}$ is an increasing sequence in $G^+$ then $\{g_n - g_{n - 1}\}$ refers to the \emph{sequence of differences}, assuming $g_0 = 0$, i.e., $g_1 - g_0 \equiv g_1$.   

\begin{proposition}\label{Prop:27}
	If $\{g_n\}$ is a truncation sequence then $\{g_n - g_{n - 1}\}$ is a good sequence, and if $\{f_n\}$ is a good sequence then $\big\{\sum_{1 \leq n \leq m} f_n\big\}$ is a truncation sequence. Thus are the truncation sequences in bijective correspondence with the good sequences.
\end{proposition}

\begin{proof}
	Given the truncation sequence $\{g_n\}$, let $g$ be the element of $\mcal{R}_0 L$ such that $g_n = g \wedge n \equiv n\ol{g/n}$ for all $n$. Then $\{\ol{g \ominus(n - 1)}\}$ is a good sequence by Lemma \ref{Lem:2}(1), and this can be expressed in the form $\{g_n - g_{n - 1}\}$ by Lemma \ref{Lem:28}(2). 
	
	Given a good sequence $\{f_n\}$, put $g_m \equiv \sum_{1 \leq n \leq m} f_n$ for all $m$. Then $f_n = \ol{g_m \ominus (n - 1)}$ for $m \geq n$ by Lemma \ref{Lem:2}(4).  Using Lemma \ref{Lem:2}(3) then yields 
	\[
	g_{m + 1} \wedge m
	=m\ol{g_{m + 1}/m} 
	= \smashoperator[lr]{\sum_{1 \leq n \leq m}} \ol{g_{m + 1}\ominus (n - 1)} 
	= \smashoperator{\sum_{1 \leq n \leq m}} f_n
	= g_m.
	\] 
	It remains to show that $\top = \bigvee_m g_m (-\infty, m)
	= \bigvee_m \big(\sum_{1 \leq n \leq m} f_n\big)(-\infty, m)$. Since each $f_n$ lies in $\ol{G}$ we may replace it by $\ol{f}_n$, yielding with the aid of Theorem \ref{Thm:4}
	\[
	\top 
	= \bigvee_m \bigvee \setof{\sbw{r}{1 \leq n \leq m} f_n(U_n)}{\smashoperator[lr]{\sum_{1 \leq n \leq m}} \ol{U}_n \subseteq (-\infty, m)}
	\]
	The displayed family is indexed by collections $\{U_n\} \subseteq \mcal{O}\mbb{R}$ such that $\sum\ol{U}_n \subseteq (-\infty, m)$. But every such family is contained in a family for which there is a particular index $j$ such that 
	\[
	U_n 
	= \begin{cases}
	(-\infty, m)    &\text{if $n = j$}\\
	(-\infty, \infty) & \text{if $n \neq j$}
	\end{cases}.
	\] 
	The term contributed by this family to the join is $\bigwedge f_n(U_n) = f_j(-\infty, m)$, so that the join over all such families works out to $\bigvee_{1 \leq j \leq m}f_j(-\infty, m)$, which reduces to $f_m(-\infty, m)$ because $\{f_m\}$ is a decreasing sequence. Finally, since $\bigwedge^\bullet f_m = 0$ we get $\bigvee_m f_m(-\infty, m) \geq \bigvee_m f_m(-\infty, 1) = \top$, as desired. 
\end{proof}

%

\subsection{Bounded truncs}\label{Subsec:BoundedTruncs}

\begin{lemma}\label{Lem:30}
	The following are equivalent for an element $g \in G^+$.
	\begin{enumerate}
		\item 
		$g$ lies in the convex $\ell$-subgroup generated by $\ol{g}$, i.e., $g \leq n\ol{g}$ for some $n$.
		
		\item 
		The truncation sequence $\{g \wedge n\}$ of $g$ is eventually constant, i.e., there exists an index $m$ such that $g \wedge n = g$ for all $n \geq m$.
		
		\item 
		There exists a real number $r > 0$ such that $\tilde{g}(x) \leq r$ for all $x \in X$.
		
		\item 
		There exists a real number $r > 0$ such that $g(-\infty, r) = \top$.
	\end{enumerate}  
\end{lemma}

\begin{proof}
	This is an application of Lemma \ref{Lem:27}. 
\end{proof}

\begin{definition*}[$G^*$, the bounded part of $G$]
	An element $g \in G$ is said to be \emph{bounded} if $\left|g\right|$ satisfies Lemma \ref{Lem:30}. The bounded elements comprise the subtrunc
	\[
	G^* \equiv \setof{g}{\text{$g$ is bounded}},
	\] 
	referred to as the \emph{bounded part of $G$.}. A trunc $G$ is said to be \emph{bounded} if $G = G^*$. 
\end{definition*}

$G^*$ is the largest bounded subtrunc of $G$.

\begin{proposition}[\cite{Ball:2014.1}, 5.1.1]
	The bounded truncs comprise a full monocoreflective subcategory $\mbf{T}^*$ of $\mbf{T}$, and a coreflector for the trunc $G$ is the insertion $G^* \to G$.    
\end{proposition}

\begin{proof}
	A truncation homomorphism $\map{\theta}{G}{H}$ clearly takes $G^*$ into $H^*$.
\end{proof}

\begin{lemma}\label{Lem:3}
	The suitable compactification of a bounded trunc $G$ is the identity map $1_M$ on its Yosida frame $M$.    
\end{lemma}

\begin{proof}
	According to the proof of Theorem \ref{Thm:3}, the suitable compactification $\map{q}{M}{L}$ of $G$ is the intersection of the open quotients of the form $g'(-\infty, \infty) \to M$, $g \in G$.  But because each $g \in G$ is bounded, each $g'(-\infty, \infty) = \top$, and each open quotient is the identity map on $M$, as is $q$. 
\end{proof}

\begin{proposition}
	The suitable compactification $\map{q}{M}{L}$ of an arbitrary trunc $G$ can be understood as realizing the inclusion $G^* \to G$. 
\end{proposition}
	
\begin{proof}
		Realizing the inclusion means finding morphisms $l$ and $m$ which make the relevant diagram from Theorem \ref{Thm:3}(b) commute.
	\begin{figure}[htb]
	\begin{tikzcd}
		G^* \arrow{r}{\mu_{G^*}} \arrow{d}
		& \mcal{E}_0 1_M \arrow{d}{\mcal{E}_0(k,l)}
		& M \arrow{r}{1_M} \arrow{d}[swap]{l}
		&M \arrow {d}{m}&\\ 
		G \arrow{r}{\mu_G} &\mcal{E}_0 q
		& M \arrow{r}[swap]{q}
		& L& 
	\end{tikzcd}	
	\end{figure}
The obvious candidates for the job are $l = 1_M$ and $m = q$.
\end{proof}

\subsection{Truncation sequences and $\mcal{E}_0 q$}\label{Subsec:E0q}

We have pointed out that for a compact pointed frame $M$, $\mcal{D}_0 M$ is closed under scalar multiplication and truncation. In particular, the n\textsuperscript{th} truncation $h \wedge n \equiv n \ol{h/n}$ lies in $\mcal{D}_0 M$ for any $h \in \mcal{D}_0^+ M$ and any $n$. This fact is easily seen from the formulas for the truncation given in Lemma \ref{Lem:27}. Proposition \ref{Prop:1} complements Lemma \ref{Lem:4}. 

\begin{proposition}\label{Prop:1}
	Let $\map{q}{M}{L}$ be a compactification. Then a given $h \in \mcal{R}_0^+ L$ lies in $\mcal{E}_0 q$ iff each truncation $h \wedge n$ factors through $q$, i.e., $h \wedge n = q \circ \hat{h}_n$ for some $\hat{h}_n \in \mcal{R}_0 M$.  
	\begin{figure}[htb]
		\begin{tikzcd}
		\mcal{O}_* \exR \arrow{r}{h'} \arrow{d}[swap]{p} & M \arrow{d}{q}\\
		\mcal{O}_*\mbb{R} \arrow{r}[swap]{h} & L
		\end{tikzcd}
		\qquad
		\begin{tikzcd}
		\mcal{O}_* \exR \arrow{r}{h' \wedge n} \arrow{d}[swap]{p} & M \arrow{d}{q}\\
		\mcal{O}_*\mbb{R} \arrow{r}[swap]{h \wedge n} \arrow {ur}[swap]{\hat{h}_n} & L
		\end{tikzcd}	
	\end{figure}
\end{proposition}

\begin{proof}
	Note first that any bounded $h' \in \mcal{D}_0^+ M$ factors through $p$ by Lemma \ref{Lem:4}. (The map $q$ in that lemma is to be understood as the identity map $M \to M$ here. The hypothesis of the lemma is satisfied because the fact that $h'$ is bounded means that $h'(-\infty, \infty) = \top$ by Lemma \ref{Lem:30}.) It follows that if $q \circ h' = h \circ p$ for some $h' \in \mcal{D}_0 M$ then the truncations of $h'$ all factor through $p$, say $h' \wedge n = \hat{h}_n \circ p$ for some $\hat{h}_n \in \mcal{R}_0 M$ and all $n$. Because the maps 
	\[
	\map{(h' \mapsto q \circ h')}{\mcal{D}_0 M}{\mcal{D}_0 L}
	\qtq{and}
	\map{(h \mapsto h \circ p)}{\mcal{R}_0 L}{\mcal{D}_0 L}
	\]
	preserve both scalar multiplication and truncation, we get that $q \circ (h' \wedge n) = (h \wedge n) \circ p$ for all $n$. It follows that $q \circ \hat{h}_n \circ p = (h \wedge n) \circ p$ for all $n$, and since $p$ is surjective, that $q \circ \hat{h}_n = h \wedge n$ for all $n$.
	
	On the other hand, suppose that we have a function $h \in \mcal{R}_0^+ L$ for which each truncation $h \wedge n$ factors through $q$, say $h \wedge n = q \circ \hat{h}_n$ for some $\hat{h}_n \in \mcal{R}_0 M$ and all $n$. Put $h_n' \equiv \hat{h}_n \circ p$, and observe that for all $n$, $h_n' = n\ol{h_{n + 1}'/n}$ because $\hat{h}_n = \ol{n\hat{h}_{n + 1}/n}$ since $h_n = n\ol{h_{n + 1}/n}$. In light of the density of the map $q$, the fact that 
	\begin{align*}
		q\bigg(\bigvee_n h_n'[-\infty, n) \bigg)
		&= q\bigg(\bigvee_n \hat{h}_n \circ p[-\infty, n)\bigg)
		= q\bigg(\bigvee_n \hat{h}_n (-\infty, n)\bigg)\\
		&= \bigvee_n q \circ \hat{h}_n(-\infty, n)
		= \bigvee_n h_n(-\infty, n)
		= \bigvee_n h(-\infty, n)\\
		& = h(-\infty, \infty)
		= \top
	\end{align*}
	 implies that $\bigvee_n h_n'[-\infty, n)$ is a dense element of $M$. Thus $\{h_n\}$ is the truncation sequence of a unique member  $h \in \mcal{D}_0 M$ by Proposition \ref{Prop:6}, a member for which $q \circ h' = h \circ p$ clearly holds. 
\end{proof}

\begin{corollary}
	For any compactification $\map{q}{M}{L}$, the bounded elements constitute a trunc in $\mcal{E}_0 q$, to be designated $\mcal{E}_0^* q$. This trunc is isomorphic to $\mcal{R}_0 M$ and to $\mcal{C}_0 X$.	
\end{corollary}

\begin{proof}
	An element $g \in \mcal{E}_0^+ q$ is bounded iff it coincides with one of its truncations, in which case it factors through $q$ by Proposition \ref{Prop:1}. That means that $g = q \circ \hat{g}$ for some $\hat{g} \in \mcal{R}_0 M$. In fact, the correspondence $\mcal{E}_0 q \to \mcal{R}_0 M = (g \mapsto q \circ \hat{g})$ is a truncation isomorphism. 
\end{proof}

\section{$\mcal{E}_0 q$ and $C^*$-embedded cozero elements of $M$\label{Sec:E0q}}

Proposition \ref{Prop:5} is the adaptation to $\mbf{T}$ of a classical result of Henriksen and Johnson \cite{HenriksenJohnson:1961}: $\mcal{D}X$ is closed under addition, i.e., $\mcal{D} X$ is a $\mbf{W}$-object, iff $X$ is a quasi-$F$ space, i.e., iff every dense cozero subset of $X$ is $C^*$-embedded. Recall that a subset $u$ is said to be \emph{$C^*$-embedded} in a space $X$ if any bounded continuous function $u \to \mbb{R}$ can be extended to a continuous function $X \to \mbb{R}$. 

\subsection{$C^*$-embedded elements of a frame}

An element $u$ is $C^*$-embedded in a frame $M$ if every frame homomorphism $\mcal{O}[0,1] \to \downset{u}$ factors through the open quotient map $M \to \downset{u} = (x \mapsto x \wedge u)$. Since $[0,1]$ is homeomorphic to $\exR$, this is equivalent to the condition that every frame homomorphism $\mcal{O} \exR \to \downset{u}$ factors through the open quotient. See \cite{BallWalters:2002} for a thorough treatment of this notion. We continue to adhere to previous notational conventions, using symbols like $\tilde{f}$ and $\tilde{g}$ for elements of $\mcal{C} X$, $\mcal{C}_0 X$ or $\mcal{D}_0 X$, and abbreviating $\tilde{f}^{-1}$ and $\tilde{g}^{-1}$ to $f'$ and $g'$. 

\begin{proposition}\label{Prop:5}
	For a compactification $\map{q}{M}{L}$, $\mcal{E}_0 q$ is a trunc iff every cozero element $u \in M$ such that $q(u) = \top$ is $C^*$-embedded.  
\end{proposition}

\begin{proof}
	Suppose the condition holds, and consider $f,g \in \mcal{E}_0^+ q$. By Proposition \ref{Prop:1} there exist $\tilde{f}, \tilde{g} \in \mcal{D}_0 X$ for which $g \circ p = q \circ \tilde{g}^{-1} = q \circ g'$ and $f \circ p = q \circ \tilde{f}^{-1} = q \circ f'$. Then $u \equiv f'(-\infty, \infty) \cap g'(-\infty, \infty)$, a cozero element of $M$ such that $q(u) = \top$, is $C^*$-embedded, hence the continuous real-valued function $\map{\tilde{f} + \tilde{g}}{u}{\mbb{\exR}}$ has a unique extension to a continuous function $\map{\tilde{h}}{X} {\exR}$. The function $h' \equiv \tilde{h}^{-1}$ drops to a function $h \in \mcal{R}_0 L$ by Lemma \ref{Lem:4} since $h'(-\infty, \infty) \geq u$, and $h$ clearly coincides with $f + g$. Thus $\mcal{E}_0 q$ is closed under addition and is therefore a trunc.
	
	Now suppose that $M$ contains a cozero element $u$ such that $q(u) = \top$, and such that $u$ is not $C^*$-embedded in $X$, say $\map{\tilde{g}} {u} {\mbb{R}}$ is a bounded continuous function with no extension to a member of $\mcal{C}X$. Note that the fact that $q(u) = \top$ implies that $* \in u$, so that, by replacing $\tilde{g}$ by $\tilde{g} - \tilde{g}(0)$ if necessary, we may assume $\tilde{g} \in \mcal{D}_0 X$. 
	
	Since $u$ is a cozero element, it is of the form $\tilde{h}^{-1}(0, \infty)$ for some $\tilde{h} \in \mcal{C}^+X$. Then $\tilde{f} \equiv 1/\tilde{h}$ is an element of $\mcal{D} X$ satisfying $u = \tilde{f}^{-1}(-\infty, \infty)$, and by replacing $f$ by $\left|\tilde{f} - \tilde{f}(0)  \right|$ if necessary, we may assume that $\tilde{f} \in \mcal{D}_0 X$. A routine verification is then enough to show that the function 
	\[
	\tilde{k}(x)
	\equiv \begin{cases}
	\tilde{f}(x) + \tilde{g}(x)    &\text{if $x \in u$}\\
	\infty         &\text{if $x \in X \smallsetminus u$} 
	\end{cases}
	\]
	lies in $\mcal{D}_0 X$. The point is that both $f'$ and $k'$ drop to functions $f,k \in \mcal{E}_0 q$ by Lemma \ref{Lem:6} because $f'(-\infty, \infty) = k'(-\infty, \infty) = u$, while their difference ($g$) does not lie in $\mcal{E}_0 q$.
\end{proof}

Proposition \ref{Prop:7} provides an instance in which Proposition \ref{Prop:5} is vacuously satisfied. An \emph{almost $P$-space} is a space having no proper dense cozero subsets (see \cite{HagervanMill:2015}). Compact examples are the one-point compactification of an uncountable discrete space, and $\beta \mbb{N} \smallsetminus \mbb{N}$. An \emph{almost $P$-frame} is a frame having no dense cozero elements other than $\top$.

\begin{proposition}\label{Prop:7}
	For a compact almost-$P$ pointed frame $M$, the only suitable compactification of the form $\map{q}{M}{L}$ is the identity map, and in this case $\mcal{E}_0 q = \mcal{R}_0 M$.
\end{proposition}
Proposition \ref{Prop:9} provides another instance to which Proposition \ref{Prop:5} applies. 

\begin{proposition}\label{Prop:9}
	A compactification $\map{q}{M}{L}$ has the feature that $\mcal{E}_0 q = \mcal{R}_0 L$ iff $q$ is the compact regular coreflection (\v{C}ech-Stone compactification) of $L$. In this case every dense cozero element of $M$ is $C^*$-embedded.
\end{proposition}

\begin{proof}
	To say that $\mcal{E}_0 q = \mcal{R}_0 L$ is to say that all truncations $g \wedge n$ of elements $g \in \mcal{R}_0^+ L$ factor through $q$. This is equivalent to the condition that all bounded elements of $\mcal{R}_0^+ L$ factor through $q$, which, by \cite[8.2.7]{BallWalters:2002}, is equivalent to $q$ being the compact regular coreflection of $L$. 
\end{proof}

We close this section by making a conjecture which frames the central question concerning the representation of truncs by means of suitable compactifications.

\begin{conjecture}\label{Con:Snug}
	Any suitable compactification $q$ admits a snugly embedded trunc $G$ in $\mcal{E}_0 q$.
\end{conjecture} 

\part{Simple truncated archimedean vector lattices\label{Part:Simple}}

In analysis, a linear combination of characteristic functions is often called a simple function. We use that term here for the corresponding elements of a trunc, which boast catchy characterizations in terms of the truncation operation (see Proposition \ref{Prop:4}). In this section we characterize those truncs composed of simple elements, culminating in Theorems \ref{Thm:1}, \ref{Thm:12}, and \ref{Thm:13}. To do so requires the representation theory of Part \ref{Part:Rep}, as well as the additional background of Section \ref{Sec:EqCat}.

\section{The equivalence of several categories\label{Sec:EqCat}}

\subsection{Idealized Boolean algebras, generalized Boolean algebras, and pointed Boolean spaces}
\begin{definition*}[idealized Boolean algebra]
	An \emph{idealized Boolean algebra} is an object of the form $(B, I)$, where $B$ is a Boolean algebra and $I$ is a maximal ideal of $B$. That is, $I$ is a proper downset in $B$ which is closed under binary joins and which contains every element or its complement. An \emph{idealized Boolean homomorphism} $\map{f}{(B, I)}{(C, J)}$ is a Boolean homomorphism $\map{f}{B}{C}$ such that $f^{-1}(J) = I$. We denote the category of idealized Boolean algebras and their homomorphisms by $\mathbf{iBa}$.
\end{definition*}

Recall the adjoint functors of Stone duality 
\begin{align*}
\map{\mathcal{S}}{\mbf{Ba}}{\mbf{zdK}}
=& (B \mapsto \st B) \\
\map{\mcal{B}}{\mbf{zdK}}{\mbf{Ba}}
=& (X \mapsto \clop X),
\end{align*}
where $\st B$ is the Boolean space of ultrafilters of the Boolean algebra $B$, and $\clop X$ is the Boolean algebra of clopen subsets of the Boolean space $X$. We extend these functors to the pointed context by modifying them as follows. 
\begin{align*}
\map{\mathcal{S}_*}{\mbf{iBa}}{\mbf{zdK}_*}
&= ((B,I)\mapsto (\uf B, B\smallsetminus I)) \\
\map{\mcal{IB}}{\mbf{zdK}_*}{\mbf{iBa}}
&= ((X, *) \mapsto (\clop X, \setof{b}{* \notin b})).
\end{align*}

\begin{proposition}
	The adjoint functors $\mcal{S}_*$ and $\mathcal{IB}$ constitute a categorical equivalence between $\mbf{iBa}$ and $\mbf{zdK}_*$.
\end{proposition}

\begin{proof}
	This is straightforward.
\end{proof}

The data required for specifying an idealized Boolean algebra $(B,I)$ is redundant, for a given maximal ideal $I$ can be a maximal ideal in only one Boolean algebra. This raises the question of the structure of a maximal ideal, by itself as a lattice. Here we show that any such ideal is a generalized Boolean algebra, and that this attribute characterizes maximal ideals as lattices. 

\begin{definition*}[generalized Boolean algebra]
	A \emph{generalized Boolean algebra} is a distributive lattice $L$ with
	designated bottom element $\bot$ satisfying 
	\[
	\forall a, b\ \exists c\ (c \vee b = a \vee b \text{ and } c \wedge
	b = \bot).
	\]
	A \emph{generalized Boolean homomorphism} is a lattice homomorphism $\map{f}{L}{M}$ which preserves the bottom element. We denote the category of generalized Boolean algebras and their homomorphisms by $\mbf{gBa}$.
\end{definition*}

Note that a generalized Boolean algebra has a designated smallest element but need not have a largest one. That is, it is closed under finite joins, including the empty join which evaluates to $\bot$, and is closed under finite nonempty meets. 

\begin{lemma}\label{Lem:18}
	For elements $a$ and $b$ in a generalized Boolean algebra $B$, there is exactly one element $c$ satisfying $c\vee b = a \vee b$ and $c \wedge b = \bot$; we denote it $a \smallsetminus b$.
\end{lemma}

\begin{proof}
	Suppose $c_i \vee b = a \vee b$ and $c_i \wedge b = \bot$ for $i=1,2$. Then $c_1 \geq c_2$ because
	\[
	c_1
	= c_1 \vee \bot 
	= c_1 \vee (c_2 \wedge b) 
	= (c_1 \vee c_2) \wedge (c_1 \vee b) 
	= (c_1 \vee c_2 ) \wedge (a \vee b) 
	= c_1 \vee c_2,
	\]
	and $c_2 \geq c_1$ dually.
\end{proof}

We remark in passing that, by taking $a \smallsetminus b$ as a primitive binary operation in addition to the lattice operations and $\bot$, the class of generalized Boolean algebras becomes a variety, i.e., the class is equationally definable. The equations are those which define distributive lattices, together with the equations mentioned in Lemma \ref{Lem:18} and the equation $a \wedge \bot = \bot$.

The forgetful functor $\map{\mcal{F}}{\mbf{iBa}}{\mbf{gBa}} = ((B, I) \mapsto I)$ provides a rich source of examples of generalized Boolean algebras, for if $I$ is a maximal ideal in a Boolean algebra $B$ then the relative complementation relation $a \smallsetminus b$ can be taken to be simply $a \wedge \lnot b$. But the functor $\mcal{F}$ is much more than a source of examples; in fact, it is an equivalence of categories. 

\begin{theorem}
	$\map{\mcal{F}}{\mbf{iBa}}{\mbf{gBa}}$ is an equivalence of categories.
\end{theorem}

\begin{proof}
	$\mcal{F}$ takes an $\mbf{iBa}$-morphism $\map{f}{(B,I)}{(C,J)}$ to its restriction $f|I$, and as such can readily be seen to be both full, i.e., surjective on morphisms, and faithful, i.e., one-one on morphisms. Consequently, we need only show that for each generalized Boolean algebra $A$ there is an idealized Boolean algebra $(B, I)$ such that $I$ is isomorphic to $A$. (See \cite[3.33, 6.8]	{AdamekHerrlichStrecker:2004}.) This is the content of Lemma \ref{Lem:19}, whose proof is a pleasant exercise in elementary lattice theory. 
\end{proof}

\begin{lemma}\label{Lem:19}
	Given a generalized Boolean algebra $A$, let $A'= \setof{a'}{a \in A}$ be a set disjoint from $A$, and let $B_A \equiv A \cup A'$. Define the Boolean operations on $B_A$ as follows.
	\[
	\begin{tabular}{|c|c|}
	\hline
	Boolean operation on $B_A$ & defined in $A \cup A'$ as \\ \hline
	$a_1 \vee a_2$ & $a_1 \vee a_2$ \\ \hline
	$a_1 \vee a_2'$ & $(a_2 \smallsetminus a_1)'$\\ \hline
	$a_1' \vee a_2'$& $(a_1 \wedge a_2)'$\\ \hline
	$a_1 \wedge a_2$ & $a_1 \wedge a_2$ \\ \hline
	$a_1 \wedge a_2'$ & $a_1 \smallsetminus a_2$ \\ \hline
	$a_1' \wedge a_2'$ & $(a_1 \vee a_2)'$ \\ \hline
	$\lnot a$ & $a'$ \\ \hline
	$\lnot a'$ & $a$ \\ \hline
	$\bot$     & $\bot$ \\ \hline
	$\top$     & $\bot'$\\ \hline
	\end{tabular}%
	\]
	With these operations $B_A$ becomes a Boolean algebra, and $A$ becomes a maximal ideal in $B$. That is, $(B_A, A)$ is an object of $\mbf{iBa}$.
\end{lemma}

In connection with Lemma \ref{Lem:19}, note that if $A$ happens to have a greatest element then it becomes a co-atom in $B_A$. 

According to \cite[6.8]{AdamekHerrlichStrecker:2004}, the forgetful functor $\mcal{F}$ must have an adjoint, and it is the functor expressed in the terms of Lemma \ref{Lem:19} by the formula 
\[
\map{\mcal{B}}{\mbf{gBa}}{\mbf{iBa}} 
= (A \mapsto (B_A, A)).
\]
We summarize.

\begin{theorem}\label{Thm:14} 
	These are categorical equivalences. 
	\[
	\mbf{zdK}_* \underset{\mcal{S}_*}{\overset{\mcal{IB}}{\rightleftarrows}}
	\mbf{iBa} \underset{\mcal{B}}{\overset{\mcal{F}}{\rightleftarrows}} \mbf{gBa}
	\]
\end{theorem}

In analysis, linear combinations of characteristic functions are often termed \enquote*{simple functions}. We use that term here for the corresponding trunc elements. (The characteristic functions themselves have a tidy characterization in terms of the truncation operator; see parts (1) and (4) of Proposition \ref{Prop:4}(1).) The simple elements comprise a subtrunc of any trunc, and it is the purpose of this section to investigate and characterize this subtrunc.

\subsection{Unital components}\label{Subsec:UnComp}

The unital components of a trunc play a prominent role in our analysis. 

\begin{definition*}[unital component]
	When speaking of elements $f,g \in G^+$, we say that \emph{$f$ is a component of $g$} if $f \leq g$ and $f \wedge (g - f) = 0$. An element $u \in \ol{G}$ is said to be a \emph{unital component of $G$} if $u \wedge g$ is a component of $g$ for each $g \in \ol{G}$. We denote the family of unital components of $G$ by $\UC(G)$, and use letters $u$, $v$, and $w$ to represent the components themselves.
\end{definition*}

The main properties of unital components are given in Proposition \ref{Prop:4}. 

\begin{proposition}\label{Prop:4}
	The following hold in a trunc $G$. 
	\begin{enumerate}
		\item 
		An element $u \in G^+$ is a unital component iff $u = \ol{2u}$.
		
		\item 
		The set $\UC(G)$ of unital components of $G$ forms a generalized Boolean algebra.
		
		\item 
		An element $u \in G^+$ serves as the unit for $GA$, i.e., $\ol{g} = g \wedge u$ for all $g \in G^+$, iff $u$ is a unital component such that $u^\perp = 0$. This happens iff $\UC(G)$ is a Boolean algebra with $u$ as greatest element.
		
		\item 
		An element $u \in G^+$ is a unital component iff $\tilde{u}$ is the characteristic function of a clopen subset of $X$ which omits the designated point $* \in X$. In symbols, $\tilde{u} = \chi_R$ for $R = \coz \tilde{u} = \tilde{u}^{-1}(0, \infty)$.
	\end{enumerate} 
\end{proposition}

\begin{proof}
	The first three parts summarize Section 3.1 of \cite{Ball:2014.1}, where proofs can be found. Part (4) follows directly from the fact that $\ol{\tilde{g}}(x) = \tilde{g}(x) \wedge 1$ for all $g \in G^+$ and $x \in X$.
\end{proof}

The unital components in $G$ are the characteristic functions of the complemented elements of $L$. This is the content of Proposition \ref{Prop:6}.  

\begin{definition*}[characteristic function $\chi_x$ in $\mcal{R} L$]
	Let $x$ be a complemented element of $L$. (That means that there is some element $y \in L$, called the \emph{complement of $x$}, such that $x \vee y = \top$ and $x \wedge y = \bot$.) The \emph{characteristic function of $x$} is the function $\chi_x \in \mcal{R} L$ defined by the rule 
	\[
	\chi_x(U)
	= \begin{cases}
	\top & \text{if $0,1 \in U$}\\
	x    & \text{if $0 \notin U \ni 1$}\\
	y    & \text{if $1 \notin U \ni 0$}\\
	\bot & \text{if $0,1 \notin U$}
	\end{cases}, \qquad U \in \mcal{O}\mbb{R}.
	\]
	If $L$ is a pointed frame then $\chi_x \in \mcal{R}_0 L$ iff $x$ does not contain the designated point $*$ of $L$, i.e., iff $*(x) = \bot$. 
\end{definition*}

\begin{proposition}\label{Prop:6}
	An element $u \in G^+$ is a unital component iff $u = \chi_x$ for some complemented element $x \in L$ such that $* \notin x$.
\end{proposition}

\begin{proof}
	If $u = \chi_x$ for complemented $x \in L$ such that $*_L \notin x$ then for $r \in \mbb{R}$ we would have  
	\[
	2u(r, \infty)
	= u(r/2, \infty)
	= \begin{cases}
	\bot    & \text{if $r \geq 2$}\\
	x       & \text{if $0 \leq r < 2$}\\
	\top    & \text{if $r < 0$}
	\end{cases},
	\] 
	so that according to Lemma \ref{Lem:8},
	\[
	\ol{2u}(r,\infty) 
	= \begin{cases}
	\bot    & \text{if $r \geq 1$} \\
	x       & \text{if $0 \leq r < 1$}\\
	\top    & \text{if $r < 0$}.
	\end{cases}.
	\]
	Evidently $\ol{2u} = u$ by inspection. On the other hand, suppose that $u = \ol{2u}$ for some $u \in G^+$. Then for all $r \in \mbb{R}$ we would have
	\[
	u(r,\infty)
	= \ol{2u}(r,\infty)
	= \begin{cases}
	\bot    & \text{if $r \geq 1$}\\
	2u(r,\infty) & \text{if $r < 1$}
	\end{cases}
	= \begin{cases}
	\bot  & \text{if $r \geq 1$}\\
	u(r/2, \infty) & \text{if $r < 1$}
	\end{cases}.
	\]  
	It follows from the fact that $u(r, \infty) = u(r/2, \infty)$ for $r < 1$ that $u(r, \infty) = u(r/2^n, \infty)$ for all $n$, hence $u(r, \infty) = \bigvee_n u(r/2^n, \infty) = u(0, \infty) = \coz u$.
	
	The proof is completed by showing that $x \equiv \coz u$ is complemented; in fact, we show that the complement of $x$ is $\con u = u(-\infty, 1)$. Surely $x \vee \con u = u(0, \infty) \vee u(-\infty, 1) = \top$; what we must show is that $x \wedge \con u = u(0, 1) = \bot$. For $0 < r < 1$ we have $u(r, \infty) = u(0, \infty) \geq u(0, r)$, which, combined with the fact that $u(0,r) \wedge u(r,\infty) = \bot$, implies $u(0,r) = \bot$. Therefore $u(0,1) = \bigvee_{0 < r < 1}u(0,r) = \bot$. 
\end{proof}

\subsection{Simple truncs}

In the next several subsections we investigate truncs determined by  their unital components. These structures have received a good deal of attention in the ordered algebra literature under the name \emph{Specker groups} (see \cite[p.\ 385] {Darnel:1994}).

\begin{definition*}[simple element, simple trunc, $\simple{G}$]
	A \emph{simple element of a trunc $G$} is a linear combination of unital components. A typical simple element thus has the form $g = \sum_U r(u) u$ for some finite subset $U \subseteq \UC(G)$ and some coefficient function $\map{r}{U}{\mbb{R}}$. (We adopt the convention that $\sum_U r(u)u = 0$ if $U = \emptyset$.) The set of simple elements is called the \emph{simple part of $G$,} written $\simple{G}$; it is the linear span of the generalized Boolean algebra $\UC(G)$ of unital components of $G$. We say that a trunc $G$ is \emph{simple} if $G = \simple{G}$. Finally, we designate the full subcategory of $\mbf{T}$ comprised of the simple truncs by $\mbf{sT}$. 
\end{definition*}

\subsection{The simple part of $\protect\widetilde{G}$}

It is easy to visualize the simple part of $\widetilde{G}$. 

\begin{proposition}\label{Prop:19}
	Let $G$ be an arbitrary trunc.
	\begin{enumerate}
		\item 
		The simple elements of $\widetilde{G}$ are the functions with finite range. 
		
		\item 
		The simple elements of $\widetilde{G}$ are the locally constant functions, i.e., the functions which, at every point, are constant on some neighborhood of the point.
		
		\item 
		$\simple G$ is a bounded subtrunc of $G$.
		
		\item 
		Every nonzero simple element can be uniquely expressed in the form $g = \sum_U r(u)u$ for a finite pairwise disjoint subset $\emptyset \neq U \subseteq \UC(G)$ and one-one function $\map{r}{U}{\mbb{R}\smallsetminus \{0\}}$. This expression is referred to as the \emph{normal form of $g$}. 
	\end{enumerate}  
\end{proposition}

\begin{proof}
	Part (1) follows readily from Proposition \ref{Prop:4}(4). Parts (2), (3), and (4) are likewise evident in $\widetilde{G}$, though the statements of (3) and (4) are in terms of $G$.
\end{proof}

It is a consequence of Proposition \ref{Prop:4}(1) that a truncation homomorphism carries unital components to unital components, and therefore simple elements to simple elements. Proposition \ref{Prop:17} follows. 

\begin{proposition}\label{Prop:17}
	The full subcategory $\mbf{sT}$ of simple truncs is monocoreflective in the category $\mbf{T}$ of archimedean truncs. A coreflector for the trunc $G$ is $\simple{G} \to G$.  
\end{proposition}

\subsection{$\protect\mbf{sT}$ is equivalent to $\protect\mbf{gBa}$}
This is Theorem \ref{Thm:15}, and in light of Theorem \ref{Thm:14}, this means that $\mbf{sT}$ is also equivalent to $\mbf{gBa}$ and $\mbf{iBa}$. The latter equivalences generalize the main result of \cite{BallMarra:2014}, which is the equivalence of the category of unital hyperarchimedean vector lattices with the category of Boolean algebras. 

The following diagram shows the relevant categories and the functors between them.
\begin{figure}[htb]
	\begin{tikzcd}
		\mbf{zdK}_* \arrow[shift left]{r}{\mcal{IB}} \arrow{d}[swap]{\mcal{LC}}
		& \bf{iBa}\arrow[shift left]{l}{\mcal{S}_*} \arrow[shift right]{d}[swap]{\mcal{F}}\\
		\mbf{sT} \arrow{r}{\UC}
		& \mbf{gBa} \arrow[shift right]{u}[swap]{\mcal{B}}
	\end{tikzcd}	
\end{figure}
Here $\mcal{LC}$ is the functor which assigns to a given Boolean pointed space $(X,*)$ the simple trunc of locally constant functions of $\mcal{D}_0 X$, i.e., $\mcal{LC}X = \simple{\mcal{D}_0 X}$. 

\begin{theorem}\label{Thm:15}
	The functor $\map{\UC}{\mbf{sT}}{\mbf{gBa}} = (G \mapsto \UC(G))$ is an equivalence of categories.
\end{theorem}

\begin{proof}
	As we mentioned prior to Proposition \ref{Prop:17}, a trunc homomorphism $\map{f}{G}{H}$ takes elements of $\UC(G)$ to elements of $\UC(H)$ and thus restricts to a $\mbf{gBa}$-morphism $\UC(G) \to \UC(H)$. Because simple truncs are generated by their unital components, different trunc homomorphisms restrict to different generalized Boolean algebra homomorphisms, i.e., $\UC$ is faithful. Furthermore, if $\map{f}{\UC(G)}{\UC(H)}$ is a $\mbf{gBa}$-morphism then we may extend $f$ to a trunc homomorphism $G \to H$ by defining $f(g) \equiv \sum_{f(U)} r(f(u))f(u)$ for elements $g \in G$ with normal form $g = \sum_U r(u)u$. That is to say that $\UC$ is full. 
	
	According to \cite[3.33]{AdamekHerrlichStrecker:2004}, it remains only to show that for each generalized Boolean algebra $A$ there exists a trunc $G$ such that $\UC(G)$ is isomorphic to $A$. But this is clear, for the idealized Boolean algebra $\mcal{B}(A) = (B_A, A)$ has the feature that its pointed Boolean space $\mcal{S}_*(B_A, A) = (X, *)$ has its clopen algebra isomorphic to $B_A$, and this isomorphism takes the clopen subsets of $X$ which omit the designated point $* \in X$ to the elements of the ideal $A \subseteq B_A$. But it is precisely these clopen subsets which correspond to the unital components of $G \equiv \mcal{LC}X$ by Proposition \ref{Prop:4}(4).  
\end{proof}

\section{Characterizing simple truncs}

\subsection{The fundamental characterization of simple truncs}

\begin{theorem}\label{Thm:1}
	Every simple trunc is isomorphic to the trunc of locally constant functions which vanish at the designated point of a unique Boolean pointed space. 
\end{theorem}

\begin{proof}
	For a simple trunc $G$, Theorem \ref{Thm:16} provides  a representation as a trunc $\widetilde{G} \subseteq \mcal{D}_0 X$ for a unique compact Hausdorff pointed space $X$ such that $\widetilde{G}$ separates the points of $X$. That $\widetilde{G} \subseteq \mcal{LC}X$ is the content of Proposition \ref{Prop:19}(2). But every locally constant function on $X$ has finite range, and is therefore a linear combination of characteristic functions, each of which is of the form $\tilde{u}$ for $u \in \UC(G)$. 
\end{proof}

In the following two subsections we characterize simple truncs in various ways, culminating in Theorems \ref{Thm:12} and \ref{Thm:13}. We begin by showing that the simple truncs are the truncs bounded away from $0$.

\subsection{Truncs bounded away from $0$}

\begin{proposition}\label{Prop:12}
	The following are equivalent for an element $0 <g \in G$.
	\begin{enumerate}
		\item 
		$\ol{ng} \in \UC(G)$ for a positive integer $n$.
		
		\item 
		$u/n \leq \ol{g} \leq u$ for some $u \in \UC(G)$ and positive integer $n$.
		
		\item 
		There is a real number $\varepsilon > 0$ such that $\tilde{g}(x) \geq \varepsilon$ whenever $\tilde{g}(x) > 0$. 
		
		\item 
		There is a real number $\varepsilon > 0$ for which $\coz g = g(0, \infty) = g(\varepsilon, \infty)$.
		
		\item 
		There is a real number $\varepsilon > 0$ for which $g(0, \varepsilon) = \bot$. 
	\end{enumerate}   
\end{proposition}

\begin{proof}
	The equivalence of the first three conditions in $\widetilde{G}$ is clear. The equivalence of (4) with (5) is likewise easy to see. For (4) implies $g(0, \varepsilon) \leq g(0, \infty) = g(\varepsilon, \infty)$, and since $g(0, \varepsilon) \wedge g(\varepsilon, \infty) = \bot$, (5) follows. And (5) implies that $g(0, \infty) = g(1, \varepsilon) \vee g(\varepsilon/2, \infty) = g(\varepsilon/2, \infty)$, i.e., (4) holds. It remains to show the equivalence of (1) with (4). 
	
	Assume (1) to prove (4), say $\ol{ng} = u \in \UC(G)$ with $x \equiv \coz u$ complemented in $L$.  Then  
	\[
	\ol{ng}(r, \infty)
	= u(r,\infty)
	= \begin{cases}
	\top    & \text{if $r < 0$}\\
	x       & \text{if $0 \leq r < 1$}\\
	\bot    & \text{if $r \geq 1$}
	\end{cases}.
	\]
	By Lemma \ref{Lem:8} we have $\ol{ng}(r, \infty) = \bot$ for $r \geq 1$ and $\ol{ng}(r, \infty) = ng(r, \infty) = g(r/n, (\infty)$ for $r < 1$. Consequently, for $r = 1/2$ we get $g(1/(2n), \infty) = \ol{ng}(1/2, \infty) = x = \coz g$. 
	
	Assume (4) to prove (1), i.e., assume $g(\varepsilon, \infty) = \coz g \equiv x$ for some $\varepsilon > 0$, and let $n$ be a positive integer such that $1/n < \varepsilon$. Then for any $t$, $0 \leq t < 1/n$, we have  
	\[
	x
	= g(0, \infty)
	\geq g(t, \infty)
	\geq g(1/n, \infty)
	\geq g(\varepsilon, \infty)
	= x.
	\]
	From this fact we can deduce using Lemma \ref{Lem:8} that 
	\begin{align*}
	\ol{ng}(s, \infty)
	&= \begin{cases}
	\bot          & \text{if $s \geq 1$}\\
	ng(s, \infty) & \text{if $s < 1$}
	\end{cases}
	= \begin{cases}
	\bot          & \text{if $s \geq 1$}\\
	g(s/n, \infty) & \text{if $s < 1$}
	\end{cases}
	= \begin{cases}
	\bot       & \text{if $s \geq 1$}\\
	x          & \text{if $0 \leq s < 1$}\\
	\top       & \text{if $s < 0$}
	\end{cases}\\
	&= \chi_x(s, \infty).
	\end{align*}
	Proposition \ref{Prop:6} then shows that $\ol{ng} \in \UC(G)$.  
\end{proof}

\begin{definition*}[bounded away from $0$, clearance]
	An element $0 < g \in G$ is said to be \emph{bounded away from $0$} if it satisfies the conditions of Proposition \ref{Prop:12}. The \emph{clearance} of such an element is  
	\[
	\clr{g}
	\equiv \bigvee \setof{\varepsilon}{g(\varepsilon, \infty) = \coz g}
	=\bigvee \setof{\varepsilon}{g(0, \varepsilon) = \bot}.
	\]
	(Our convention is that $\clr{0} = 0$.) Finally, a trunc $G$ is said to be \emph{bounded away from $0$} if every $0< g \in G$ is bounded away from $0$. 
\end{definition*} 

\begin{corollary}\label{Cor:1}
	If $0 < g \in G$ is bounded away from $0$ then $\coz g \equiv x$ is complemented in $L$ and $\chi_x \in \UC(G)$. 
\end{corollary}

When checking whether a trunc is bounded away from $0$, it is enough to verify that the elements of $\ol{G}$ are bounded away from $0$

\begin{lemma}\label{Lem:16}
	An element $0 < g \in G$ is bounded away from $0$ iff $\ol{g}$ is bounded away from $0$. 
\end{lemma}

\begin{proof}
	According to Lemma \ref{Lem:8}, $\ol{g}(r, \infty) = g(r, \infty)$ for $r < 1$. It follows that there is a real number $\varepsilon > 0$ for which $g(0, \varepsilon) = \bot$ iff there is a real number $\varepsilon > 0$ for which $\ol{g}(0, \varepsilon) = \bot$.
\end{proof}

The simple part of any trunc is bounded away from $0$. 

\begin{proposition}\label{Prop:13}
	A strictly positive simple element is bounded away from $0$. Thus a simple trunc is bounded away from $0$.
\end{proposition}

\begin{proof}
	If $g$ is a simple element then $\tilde{g}$ has finite range by Proposition \ref{Prop:19}(1). If the element is positive then the range has a least positive element, and the element is bounded away from $0$ by Proposition \ref{Prop:12}.  
\end{proof}

Proposition \ref{Prop:13} has a converse in Theorem \ref{Thm:12}. What follows is a sequence of lemmas which together constitute a proof of that theorem. In these lemmas we fix our attention on an element $0 < g \in \ol{G}$ of a trunc $G$ which is bounded away from $0$. We abbreviate $\coz g$ to $x$ and $\clr{g}$ to $\delta$; by Corollary \ref{Cor:1}, $x$ is complemented in $L$ and $\chi_x =u$ for a unique $u \in \UC(G)$. Note that $x > \bot$ and $u \geq g > 0$. 

\begin{lemma}\label{Lem:11}
	For real numbers $s$ and $r > 0$, 
	\[
	g \ominus r(s, \infty)
	= \begin{cases}
	\top     & \text{if $s < 0$}\\
	g(s + r, \infty)     & \text{if $s \geq 0$}.
	\end{cases}
	\] 
\end{lemma}

\begin{proof}
	Making use of Lemma \ref{Lem:8}, we get
	\begin{align*}
	g \ominus r(s, \infty)
	&= r((g/r) \ominus 1)(s, \infty)
	=  (g/r) \ominus 1 (s/r, \infty)\\
	&= \begin{cases}
	\top     & \text{if $s/r < 0$}\\
	(g/r)(s/r + 1, \infty)   & \text{if $s/r \geq 0$}
	\end{cases} 
	= \begin{cases}
	\top    & \text{if $s < 0$} \\
	g(s + r, \infty)  & \text{if $s \geq 0$}
	\end{cases}.\qedhere
	\end{align*}
\end{proof}

\begin{lemma}\label{Lem:2}
	\begin{enumerate}
		\item 
		$\coz g \ominus \delta < x$.
		
		\item 
		$g = \delta u$ iff $g \ominus \delta = 0$.		
	\end{enumerate}
\end{lemma}

\begin{proof}
	(1) Suppose for the sake of argument that $\coz g \ominus \delta = x$. Since $g \ominus \delta > 0$ is bounded away from $0$ and strictly positive, there exists $\varepsilon > 0$ for which $g \ominus \delta(\varepsilon, \infty) = g \ominus \delta(0, \infty) = x$. But since $g \ominus \delta(\varepsilon, \infty) = g(\varepsilon + \delta, \infty)$ by Lemma \ref{Lem:11}, we arrive at the contradiction $\varepsilon + \delta \leq \delta$.   
	
	(2) If $g = \delta u = \delta\chi_x$ then for all $s \in \mbb{R}$,
	\begin{align*}
	g(s, \infty)
	&= \delta u(s, \infty)
	= u(s/\delta, \infty)
	= \begin{cases}
	\top    & \text{if $s < 0$} \\
	x       & \text{if $0 \leq s/\delta < 1$}\\
	\bot    & \text{if $s/\delta \geq 1$}
	\end{cases}
	=  \begin{cases}
	\top    & \text{if $s < 0$} \\
	x       & \text{if $0 \leq s < \delta$}\\
	\bot    & \text{if $s \geq \delta$}
	\end{cases}.
	\end{align*}
	Combining this with Lemma \ref{Lem:11} yields that $g \ominus \delta (s, \infty) = \top$ if $s < 0$ and $g \ominus \delta(s, \infty) = \bot$ if $s \geq 0$, which is to say that $g \ominus \delta = 0$ in $G$.
	
	If $g \ominus \delta = 0$ then Lemma \ref{Lem:11} tells us that 
	\[
	g \ominus \delta(s, \infty)
	= \begin{cases}
	\top    & \text{if $s < 0$}\\
	g(s + \delta, \infty) & \text{if $s \geq 0$}
	\end{cases}
	= 0(s, \infty)
	= \begin{cases}
	\top   & \text{if $s < 0$}\\
	\bot   &  \text{if $s \geq 0$}
	\end{cases}.
	\]
	In light of the fact that $g(s, \infty) = x$ for $0 \leq s < \delta$, we get 
	\begin{align*}
	g(s, \infty)
	&= \begin{cases}
	\top     & \text{if $s < 0$}\\
	x        & \text{if $0 \leq s < \delta$}\\
	\bot     & \text{if $s \geq \delta$}
	\end{cases}\
	= \delta u. \qedhere
	\end{align*}
\end{proof}

Note that $g \ominus \delta$ is bounded away from $0$, and if it is strictly positive then by Corollary \ref{Cor:1} we know that $y \equiv \coz g \ominus \delta$ is complemented in $L$ and $\chi_y \equiv w \in \UC(G)$. Now $w \leq u$ so $v \equiv u - w$ is a unital component such that $w \vee v = u$ and $w \wedge v = 0$. Since $g \leq u$, we can express $g$ in the form $g = g_w + g_v$ for $g_w \equiv g \wedge w$ and $g_v = g \wedge v$. 

\begin{lemma}
	$g_v = \delta v$.
\end{lemma}

\begin{proof}
	Let $z \equiv \coz v$, so that $y \vee z = x$ and $y \wedge z = \bot$. First note that $g(s, \infty) = g(0,\infty) = x$ for $0 \leq s < \delta$, and that by Lemma \ref{Lem:11},
	\[
	g(\delta, \infty) \wedge z 
	= g \ominus \delta(0, \infty) \wedge z 
	= y \wedge z 
	= \bot.
	\] 
	It then follows that for any $s \in \mbb{R}$, 
	\begin{align*}
	g_v(s, \infty) 
	&= g \wedge v)(s, \infty)
	= g(s, \infty) \wedge v(s, \infty)\\
	&= g(s, \infty) \wedge \begin{cases}
	\top     & \text{if $s < 0$}\\
	z        & \text{if $0 \leq s < 1$}\\
	\bot     & \text{if $s \geq 1$}	\end{cases}
	= \begin{cases}
	\top    & \text{if $s < 0$}\\
	z       & \text{if $0 \leq s < \delta$} \\
	\bot    & \text{if $s \geq \delta$} \end{cases}
	= (\delta v)(s, \infty). \qedhere
	\end{align*}
\end{proof}

\begin{lemma}
	$\clr{g_w} = \delta + \clr{g \ominus \delta} > \delta$.
\end{lemma}

\begin{proof}
	For $s \geq \delta$ we have $g_v(s, \infty) = \delta v(s, \infty) = v(s/\delta, \infty) = \bot$, hence 
	\[
	g_w(s, \infty)
	= g_w(s, \infty) \vee g_v(s, \infty)
	= (g_w \vee g_v)(s, \infty)
	= g(s, \infty)
	= g \ominus \delta(s - \delta, \infty)
	\]
	by Lemma \ref{Lem:11}. Since $\coz g_w = \coz g \ominus \delta = y$, it follows from the equation above that $g_w(s, \infty) = \coz g_w$ iff $s - \delta < \delta(g \ominus \delta)$, which is to say that $\delta(g_w) = \delta + \delta(g \ominus \delta)$. 
\end{proof}

Proposition \ref{Prop:14} summarizes the development to this point.

\begin{proposition}\label{Prop:14}
	Let $G$ be a trunc which is bounded away from $0$. Then for each $0 < g \in \ol{G}$ there exist unique disjoint elements $g_1 \in \ol{G}$ and $0 < u \in \UC(G)$ such that $g = g_1 + \clr{g} u$ and $\clr {g_1} > \clr{g}$ if $g_1 > 0$.
	
\end{proposition}

\begin{proof}
	In terms of the preceding lemmas, take $g'$ to be $g_w$ and $u$ to be $v$. 
\end{proof}

\begin{theorem}\label{Thm:12}
	A trunc is simple iff it is bounded and bounded away from $0$.
\end{theorem}

\begin{proof}
	A simple trunc is certainly bounded, and it is bounded away from $0$ by Proposition \ref{Prop:13}. Now suppose that $G$ is bounded and bounded away from $0$. Since $G$ is bounded, each element is a linear combination of finitely many members of $\ol{G}$ (see Subsection \ref{Subsec:GoodSeq} on good sequences), so that to show $G$ simple it is enough to show that each element of $\ol{G}$ is simple. 
	
	For that purpose consider $0 < g_0 \in \ol{G}$, and let $g_1 \in \ol{G}$ and $0 < u_1 \in \UC(G)$ be the disjoint elements given by Proposition \ref{Prop:14} such that $g_0 = g_1 + \clr{g_0} u_1$ and $\clr{g_1} > \clr{g_0}$ if $g_1 > 0$.  Proceed inductively. If $g_n$ and $u_n$ have been defined such that $g_n > 0$ then let $g_{n + 1} \in \ol{G}$ and $0 < u_{n + 1} \in \UC(G)$ be the disjoint elements which satisfy $g_n = g_{n + 1} + \clr{g_n} u_{n + 1}$ and $\clr{g_{n + 1}} > \clr{g_n}$ if $g_{n + 1} > 0$. The induction continues as long as $g_{n + 1} > 0$, and terminates if $g_{n + 1} = 0$. Note that if $g_n$ is defined then   
	\[
	g_0 
	= g_n + \smashoperator{\sum_{0 \leq i < n}}\clr{g_i} u_{i + 1},
	\] 
	$\{\clr{g_i}\}$ is a strictly increasing sequence of positive real numbers bounded above by $1$, and the $u_i$'s are nonzero pairwise disjoint unital components bounded above by the element $u_0 \in \UC(G)$ given by Proposition \ref{Prop:12}(2) such that $u_0/n \leq \ol{g_0} \leq u_0$ for some positive integer $n$.
	
	The proof is completed by showing that the sequence of $g_n$'s is finite. If not, let $r \equiv \bigvee_n \clr{g_n}$, and consider the element $h \equiv (ru_0 - g_0)^+$. We aim to show that $h(0, \varepsilon) > \bot$ for any  $\varepsilon > 0$, thereby showing that $h$ violates condition (5) of Proposition \ref{Prop:12} and thus contradicts the hypothesis that $G$ is bounded away from $0$. For that purpose fix $\varepsilon > 0$, let $n$ be such that $r < \clr{g_n} + \varepsilon$, abbreviate $\clr{g_n}$ to $s$, and let $z = \coz u_{n + 1}$. Since $g_0 \geq su_{n + 1}$, we can compute with the aid of Theorem \ref{Thm:4} 
	\begin{align*}
	h(-\infty, \varepsilon)
	&= (r u_0 - g_0)^+(-\infty, \varepsilon)
	\geq (r u_0 - s u_{n + 1})^+(-\infty, \varepsilon)
	= \sbv{lr}{rU - sV \subseteq (-\infty, \varepsilon)}(\chi_x(U) \wedge \chi_z(V)).
	\end{align*}
	The open sets $U \equiv (1 - \varepsilon/2r, 1 + \varepsilon/2r)$ and $V \equiv (1 - \varepsilon/2s, 1 + \varepsilon/2s)$ clearly satisfy the condition $rU -sV \subseteq (-\infty, \varepsilon)$, so $\bot < z 
	\leq \chi_x(U) \wedge \chi_z(V) \leq h(-\infty, \varepsilon)$. On the other hand, we have 
	\begin{align*}
	h(0, \infty)
	&= (ru_0 - g_0)(0, \infty)
	= \smashoperator{\bigvee_{rU -V \subseteq (0, \infty)}}(\chi_x(U) \wedge g_0(V)),
	\end{align*}
	so that if we put $U \equiv (1 - \rho/r, 1 + \rho/r)$ and $V \equiv (s - \rho, s + \rho)$ for $\rho < (\varepsilon + s - r)/2$ then we can see that $rU -V \subseteq (0, \infty)$, with the result that 
	\[
	h(0, \infty)
	\geq \chi_x(U) \wedge g_0(V)
	\geq \chi_x(U) \wedge su_{n + 1}(V)
	\geq z.
	\]
	To summarize, $h(0, \varepsilon) = h(-\infty, \varepsilon) \wedge h(0, \infty) \geq z > \bot$.
\end{proof}

\subsection{Hyperarchimedean truncs}

Hyperarchimedean vector lattices have been intensively studied; see \cite[pp.\ 380 ff.]{Darnel:1994} for an introduction. Of the many characterizations in the literature, we list here the trunc versions of the three most often mentioned.  

\begin{lemma}\label{Lem:12}
	The following are equivalent for a trunc $G$.
	\begin{enumerate}
		\item 
		Every convex $\ell$-subgroup $K \subseteq G$ is an archimedean kernel, i.e., $G/K$ is an archimedean vector lattice.
		
		\item 
		Every prime convex $\ell$-subgroup is both maximal and minimal.
		
		\item 
		For every $g \in G^+$, the convex $\ell$-subgroup $G(g)$ generated by $g$ is a cardinal summand of $G$. That means that for every $f \in G$ there exist unique elements $f_g \in G(g)$ and $f' \in g^\perp$ such that $f = f_g + f'$.
	\end{enumerate}
	A trunc which satisfies these conditions is called \emph{hyperarchimedean}.  
\end{lemma}

\begin{proof}
	A proof of these equivalences in the broader context of archimedean $\ell$-groups is part of Theorem 55.1 in \cite{Darnel:1994}.  
\end{proof}

\begin{proposition}\label{Prop:20}
	A simple trunc is hyperarchimedean.
\end{proposition}

\begin{proof}
	Consider elements $f,g > 0$ in a simple trunc $G$. Since $G$ is bounded away from $0$ by Theorem \ref{Thm:12}, there is a positive integer $n$ for which $u \equiv \ol{ng}$ is a unital component by Proposition \ref{Prop:12}(1). Hence $\ol{f} = f_u + f'$ for unique elements $f_u \leq u$ and $f' \in u^\perp$. Since $G$ is bounded, again by Theorem \ref{Thm:12}, there is an integer $k$ such that $k\ol{f} = kf_u + kf' \geq f$, and since $f_u \wedge f' = 0$, we have $f = (kf_u \wedge f) + (kf' \wedge f)$. But $kf_u \wedge f \in G(g)$ because 
	\[
	kf_u \wedge f \leq kf_u \leq ku = k\ol{ng} \leq (kn)g,
	\]  
	and $kf' \wedge f \in u^\perp = \ol{ng}^\perp = g^\perp$, so we have shown that $G$ is hyperarchimedean by criterion (3) of Lemma \ref{Lem:12}.
\end{proof}

Example \ref{Ex:1} shows that the converse of Proposition \ref{Prop:20} is false.

\begin{example}\label{Ex:1}
	Let $X$ be the pointed Boolean space $(\omega + 1, \omega)$, and let $\widetilde{G}$ be the family of all functions of $\mcal{D}_0 X$ of the form $\tilde{a} + r\tilde{g}_0$, with $r \in \mbb{R}$ and $\tilde{a}, \tilde{g}_0 \in \mcal{D}_0 X$ such that $\coz \tilde{a}$ is finite and $\tilde{g}_0(n) = 1/n$ for all $n < \omega$. Then it is straightforward to check that $G$ is a hyperarchimedean trunc which is not simple because $\tilde{g}_0$ is not bounded away from $0$.  
\end{example}

In order to identify which hyperarchimedean truncs are simple, we consider two trunc attributes. The first is the property of having enough unital components, and the second is being bounded away from infinity. In Theorem \ref{Thm:13} we show that, in the presence of the hyperarchimedean property, either of these attributes is equivalent to simplicity.

\begin{definition*}[enough unital components]
	A trunc $G$ is said to \emph{have enough unital components} if for all $g \in G^+$ there exists a unital component $u \in \UC(G)$ such that $\ol{g} \leq u$. 
\end{definition*}    

\begin{lemma}\label{Lem:13}
	The following are equivalent for an element $g \geq 0$ of a trunc $G$.
	\begin{enumerate}
		\item 
		There exists an element $h \in G^+$ for which $\ol{g} \leq h \ominus 1$.
		
		\item 
		$\tilde{g}$ vanishes on a neighborhood of $*$, i.e., $* \notin \cl \coz \tilde{g}$.
		
		\item 
		There exists $h \in G^+$ for which $\coz g \wedge \con h = \bot$. (Recall $\con h \equiv h(-\infty, 1)$.)
	\end{enumerate}
\end{lemma}

\begin{definition*}[bounded away from $\infty$]
	We say that $g$ is \emph{bounded away from $\infty$} if it satisfies the conditions of Lemma \ref{Lem:13}. We say that \emph{$G$ is bounded away from $\infty$} if each $g \in G^+$ is bounded away from $\infty$.
\end{definition*}

\begin{proof}[Proof of Lemma \ref{Lem:13}.]
	In the Yosida representation of $G$, the open subsets of $X$ of the form $\con \tilde{h} = \tilde{h}^{-1}(-\infty, 1)$, $h \in G^+$, form a neighborhood base for the designated point. Thus (2) is equivalent to the existence of an element $h \in G^+$ for which $\coz \tilde{g} \cap \con \tilde{h} = \emptyset$. A routine calculation then shows the latter condition to be equivalent to $\coz g \wedge \con h = \bot$, thereby establishing the equivalence of (2) with (3). 
	
	To show the equivalence of (1) with (2), note that if $\ol{g} \leq h \ominus 1$ then 
	\begin{align*}
	x \in \coz \tilde{g} = \coz \widetilde{\ol{g}}
	&\implies 1 = \ol{g}(x) \leq (\tilde{h}\ominus 1)(x) = (\tilde{h}(x) - 1) \vee 0\\
	&\implies \tilde{h}(x) = 2 \implies x \notin \con \tilde{h}.
	\end{align*}
	On the other hand, if $\tilde{g}$ vanishes on a neighborhood of $*$ then, since the sets of the form $\con \tilde{h}$, $h \in G^+$, form a neighborhood base for $*$, there is some $h \in G^+$ such that $\coz \tilde{g} \wedge \con \tilde{h} = \emptyset$. For such an $h$, it is straightforward to check that $2\tilde{h} \ominus 1 \geq \widetilde{\ol{g}}$. 
\end{proof}

\begin{lemma}\label{Lem:14}
	\begin{enumerate}
		\item 
		If $G$ has enough unital components then $G$ is bounded away from $\infty$.
		
		\item 
		If $G$ is hyperarchimedean and bounded away from $\infty$ then $G$ has enough unital components.
	\end{enumerate}
\end{lemma}

\begin{proof}
	(1) Suppose $g \in G^+$ is such that $\ol{g} \leq u \in \UC(G)$. Then $\coz g = \coz \ol{g} \leq \coz u$, and since $\coz u \wedge \con u = \bot$, it follows that $\coz g \wedge \con u = \bot$. We conclude that $g$ is bounded away from $\infty$ by Lemma \ref{Lem:13}. 
	
	(2) Suppose $\ol{g} \leq h \ominus 1$ for $g,h \in G^+$. This is equivalent to the assertion that $\tilde{h}(x) > 1$ for all $x \in \coz \tilde{g}$, hence $\widetilde{\ol{h}}(x) = 1$ for all $x \in \coz \tilde{g}$. If $G$ is hyperarchimedean then it is the cardinal sum $G(g) \oplus g^\perp$, so that $\ol{h}$ can be uniquely expressed in the form $u + h'$ for $u \in G(g)$ and $h' \in g^\perp$. Since $\tilde{u}(x) = 1$ for $x \in \coz \tilde{g}$ and $\tilde{u}(x) = 0$ for $x \notin \coz \tilde{g}$, we have $g \leq u \in \UC(G)$.   
\end{proof}

\begin{theorem}\label{Thm:13}
	The followiing are equivalent for a trunc $G$.
	\begin{enumerate}
		\item 
		$G$ is hyperarchimedean and has enough unital components.
		
		\item 
		$G$ is hyperarchimedean and bounded away from $\infty$.
		
		\item 
		$G$ is simple.
	\end{enumerate}
\end{theorem}

\begin{proof}
	The equivalence of (1) and (2) is a consequence of Lemma \ref{Lem:14}. The implication from (3) to (2) follows from Proposition \ref{Prop:20}, together with the observation that each element $g \geq 0$ of a simple trunc $G$ has the feature that $\tilde{g}$ is locally constant. The point is that since $\tilde{g}$ is $0$ at the designated point $*$, it must then be $0$ on a neighborhood of $*$, which is to say that $g$ is bounded away from $\infty$. 
	
	It remains to show that a trunc $G$ which satisfies (1) also satisfies  (3); by Theorem \ref{Thm:12}, it is enough to show that such a trunc is bounded and bounded away from $0$. It is certainly bounded, for given any $g \in G^+$, the hyperarchimedean property means that $G$ is the cardinal sum of $G(\ol{g})$ and $\ol{g}^\perp = g^\perp$, hence $g \leq n\ol{g}$ for some positive integer $n$. In order to show that an arbitrary $g \in G^+$ is bounded away from $0$ we may assume that $g = \ol{g}$ by Lemma \ref{Lem:16}. Since $G$ has enough unital components, there exists $u \in \UC(G)$ such that $g \leq u$. Since $G$ is the cardinal sum of $G(g)$ and $g^\perp$, we can write $u$ in the form $u_g + u'$ for $u_g \in G(g)$ and $u' \in g^\perp$. It follows that $g \leq u_g \in \UC(G)$ and $u_g \leq ng$ for some positive integer $n$. In fact we have $u_g/n \leq g \leq u_g \in \UC(G)$, so that $g$ is bounded away from $0$ by Proposition \ref{Prop:12}(2).   
\end{proof}

We should point out that a truncation homomorphism $\map{\theta}{G}{H}$ takes elements of $G$ bounded away from $\infty$ to elements of $H$ bounded away from $\infty$. For by Theorem \ref{Thm:16}, such a map is realized by a continuous function $\map{f}{Y}{X}$, where $X$ and $Y$ designate the pointed Yosida spaces of $G$ and $H$, respectively, in the sense that $\widetilde{\theta(g)} = \tilde{g}\circ f$. Since $f$ takes the designated point $*_Y$ of $Y$ to the designated point $*_X$ of $X$, it follows that if $\tilde{g}$ vanishes on a neighborhood of $*_X$ then $\widetilde{\theta(g)}$ vanishes on a neighborhood of $*_Y$. This observation explains Proposition \ref{Prop:15}.

\begin{proposition}\label{Prop:15}
	The truncs bounded away from $\infty$ comprise a full monocoreflective subcategory of the category $\mbf{T}$ of archimedean truncs. A coreflector for a trunc $G$ is the insertion of the subtrunc of elements bounded away from $\infty$. 
\end{proposition} 

	We remark that in any trunc $G$, the elements bounded away from $\infty$ form a convex subtrunc
	\[
	K \equiv \setof{g}{\text{$\left|g\right|$ is bounded away from $\infty$}},
	\] 
	and if $G$ is hyperarchimedean then $K$ is also an archimedean kernel, i.e., $G/K$ is archimedean. However, $K$ is not always a truncation kernel in the sense of Section \ref{Sec:TruncKer}, for it fails to satisfy requirement (3) of Lemma \ref{Lem:20} below. Indeed, this is the case in Example \ref{Ex:1}.

\part{Truncation kernels\label{Part:TruncKer}}

We conclude with a brief discussion of truncation kernels. The topic is of intrinsic interest in any study of truncs, of course, but the discussion acquires a degree of urgency by virtue of the necessity of correcting a serious error in \cite{Ball:2014.2}.  

\section{Truncation kernels\label{Sec:TruncKer}}

\begin{definition*}[$\mbf{T}$-kernel]
	A \emph{truncation kernel}, or $\mbf{T}$-kernel, is the set of elements of a trunc sent to $0$ by a $\mbf{T}$-morphism.  
\end{definition*}

Let $\mbf{U}$ be the category of not-necessarily-archimedean vector lattices. The distinction between $\mbf{U}$-kernels and $\mbf{T}$-kernels is an important one, for the former have the feature that every proper kernel is contained in a maximal proper kernel, while the latter lack this feature. In fact, it can be shown that the maximal proper $\mbf{U}$-kernels are in bijective correspondence with the points of the Yosida space of a trunc (Subsection \ref{Subsec:YosRepTruncs}), while the $\mbf{T}$-kernels are in bijective correspondence with the elements of the Madden frame of a trunc (Subsection \ref{Subsec:MaddenRep}).

\subsection{Correcting a basic lemma}

We correct an error in Lemma 2.1.2 of \cite{Ball:2014.2}. That lemma is missing an important hypothesis; the corrected version appears below as Lemma \ref{Lem:20}, in which the missing hypothesis is part (3). Archimedean truncation kernels are further characterized in Proposition \ref{Prop:30}. 

\begin{lemma}\label{Lem:20}
	A convex subtrunc $K \subseteq G$ is a truncation kernel iff it satisfies properties (2) and (3) below; it is an archimedean truncation kernel iff it also satisfies (1).
	\begin{enumerate}
		\item 
		If there exists $h \in G^+$ such that $(ng - h)^+ \in K$ for all $n$ then $g \in K$.
		\item 
		If $\ol{g} \in K$ then $g \in K$.
		
		\item 
		If $g \ominus 1/n \in K$ for all $n$ then $g \in K$. 
	\end{enumerate}
\end{lemma}

\begin{proof}
	(1) is well known to be equivalent to the archimedean property of the quotient $G/K$, and (2) is clearly equivalent to truncation property ($\mfrak{T}2$) of the quotient. We claim that (3) is equivalent to truncation property ($\mfrak{T}3$) of the quotient. For $g \ominus 1/n = (ng \ominus 1)/n = (ng - \ol{ng})/n$, so that the condition that $g \ominus 1/n \in K$ for all $n$ is equivalent to the condition that $K + ng = K + \ol{ng}$ for all $n$. 	
\end{proof}

The falsity of \cite[2.1.2]{Ball:2014.2} does not invalidate the subsequent results of \cite{Ball:2014.2}. For example, in the proof of \cite[2.3.4]{Ball:2014.2}, it is straightforward to verify that the missing condition (3) of Lemma \ref{Lem:20} above is satisfied for the set displayed on the right side. Likewise, the internal description of the archimedean truncation kernel $[K]$ generated by a subset $K \subseteq G$ can be readily modified to take condition (3) into account, as follows. 

Every ordinal number $\alpha$ can be expressed in the form $\alpha = \beta + k$ for a unique finite ordinal $k$ and limit ordinal $\beta$. (We take $\beta = 0$ to be a limit ordinal.). We will say that \emph{$\alpha$ is congruent to $i$ mod $3$}, and write $\alpha \equiv i \mod 3$, depending on whether $k$ is congruent to $i$ mod $3$.

For a subset $K \subseteq G$, let $\langle K \rangle$ designate the
convex $\ell$-subtrunc generated by $K$. Now define
\begin{align*}
	K^0  
	&\equiv \langle K\rangle,\\
	K^{\alpha+1}  
	&\equiv \langle g \in G^+ : \forall n\ (g \ominus 1/n \in K^\alpha)\rangle \text{ if $\alpha \equiv 0 \mod 3$,}\\
	K^{\alpha+1}  
	&\equiv \langle g \in G^+ : \exists h \in G^+\ \forall n\ ((n\left|g\right| - h)^+ \in K^\alpha)\rangle \text{ if $\alpha \equiv 1 \mod 3$,}\\
	K^{\alpha+1}  
	&\equiv \langle g \in G^+ : \ol{g} \in K^\alpha\rangle \text{ if $\alpha \equiv 2 \mod 3$,}\\
	K^\beta  
	&\equiv \bigcup_{\alpha<\beta} K^\alpha\text{ if $\beta$ is a limit ordinal,}\\
	K^\infty  
	&\equiv K^\alpha\text{ for some (any) $\alpha$ such that $K^\alpha = K^{\alpha+1} = K^{\alpha + 2}$.}
\end{align*}

\begin{lemma}\label{Lem:24}
	The archimedean truncation kernel generated by a subset $K \subseteq G$, which we shall denote by $[K]$, is equal to $K^{\infty} = K^{\omega_{1}}$ .
\end{lemma}

\begin{proof}
	This follows directly from Lemma \ref{Lem:20} above.
\end{proof}

With these and similar minor modifications, the proofs given in \cite{Ball:2014.2} become correct. 

\subsection{Pointwise closure and archimedean truncation kernels}
Archimedean truncation kernels are characterized by the property of being pointwise closed. This is not surprising in view of the fact that the same is true of $\mbf{W}$-kernels (\cite[5.3.1]{BallHagerWW:2015}).
 
\begin{definition*}[pointwise closed convex subtrunc]
	A convex subtrunc $K \subseteq G$ is said to be \emph{pointwise closed} if $K_0 \subseteq K^+$ and $\bigvee^\bullet K_0 = g$ imply $g \in K$.   
\end{definition*}

\begin{proposition}[cf.\ \cite{BallHagerWW:2015}, 5.3.1]\label{Prop:30}
	A convex subtrunc $K \subseteq G$ is a truncation kernel iff it is pointwise closed.
\end{proposition}

\begin{proof}
	Suppose that $K$ is a truncation kernel, and let $\map{\theta}{G}{G/K}$ be the quotient truncation homomorphism. If $K_0 \subseteq K^+$ and $\bigvee^\bullet K_0 = g$ then $\bigvee \theta(K_0) = \theta(g) = 0$ by Proposition \ref{Prop:21}, hence $g \in K_0$. Thus $K$ is pointwise closed.
	
	Now suppose that $K$ is a pointwise closed convex subtrunc of $G$; we must show that $K$ has the three properties of Lemma \ref{Lem:20}. The proof given in \cite[5.3.1]{BallHagerWW:2015} for (1) in $\mbf{W}$ works without modification in $\mbf{T}$. To check (2) and (3), recall that $n\ol{g/n} \nearrow g$ and $g \ominus(1/n) \nearrow g$ for all $g \in G^+$ by \cite[5.5]{Ball:2018}.
\end{proof}

%
%
%

\end{document}